\documentclass[reqno,11pt]{amsart}
\usepackage{amsmath,amssymb,graphicx,amsthm,fullpage,xcolor,mathrsfs}
\usepackage{stmaryrd}

\usepackage{xifthen,xcolor,tikz,setspace}
\usetikzlibrary{decorations.pathmorphing,patterns,shapes,calc}

\usepackage[normalem]{ulem}

\usepackage{prettyref}
\usepackage{bbm}
\usepackage{enumerate}

\usepackage[colorlinks=true]{hyperref}
\colorlet{DarkGreen}{green!50!black}
\colorlet{DarkGray}{gray!60!black}

\DeclareMathAlphabet{\mathcal}{OMS}{cmsy}{m}{n}

\newtheorem*{conjecture*}{Conjecture}
\newtheorem{theorem}{Theorem}[section]
\newtheorem*{theorem*}{Theorem}
\newtheorem{lemma}[theorem]{Lemma}
\newtheorem*{lemma*}{Lemma}

\newtheorem{proposition}[theorem]{Proposition}

\newtheorem{corollary}[theorem]{Corollary}

\theoremstyle{definition}{
\newtheorem{example}{Example}
\newtheorem{definition}[theorem]{Definition}
\newtheorem*{definition*}{Definition}

\newtheorem{remark}[theorem]{Remark}

}

\numberwithin{equation}{section}

\newcommand{\cM}{\mathscr{M}}
\newcommand{\cN}{\mathcal{N}}

\newcommand{\cF}{\mathcal{F}}
\newcommand{\cG}{\mathcal{G}}

\newcommand{\cS}{\mathcal{S}}

\newcommand{\R}{\mathbb{R}}

\newcommand{\prob}{\mathbb{P}}
\newcommand{\E}{\mathbb{E}}

\newcommand{\eps}{\epsilon}

\newcommand{\eqdist}{\stackrel{(d)}{=}}
\newcommand{\convdist}{\stackrel{(d)}{\rightarrow}}
\newcommand{\tensor}{\otimes}

\newcommand{\abs}[1]{\left\lvert#1\right\rvert}
\newcommand{\norm}[1]{\left\lVert#1\right\rVert}

\newcommand{\cov}{\operatorname{Cov}}
\newcommand{\vol}{\operatorname{vol}}
\newcommand{\tr}{\operatorname{tr}}

\newcommand{\g}[1]{\left\langle{#1}\right\rangle}

 \newcommand{\CI}{\textbf{Condition~1}} 
 \newcommand{\CII}{\textbf{Condition~2}}
 \newcommand{\CIIprime}{\textbf{Condition~2'}}

\renewcommand{\limsup}{\varlimsup}
\renewcommand{\liminf}{\varliminf}
\newrefformat{prop}{Proposition~\ref{#1}}
\newrefformat{fig}{Figure~\ref{#1}}
\newrefformat{cor}{Corollary~\ref{#1}}
\newrefformat{conj}{Conjecture~\ref{#1}}
\newrefformat{rem}{Remark~\ref{#1}}

\begin{document}
\title{Algorithmic thresholds for Tensor PCA}

\author{G\'erard Ben Arous}
\author{Reza Gheissari}
\author{Aukosh Jagannath}

\address[G\'erard Ben Arous]{Courant Institute, New York University and NYU Shanghai}
\email{benarous@cims.nyu.edu}

\address[Reza Gheissari]{Courant Institute, New York University and UC Berkeley}
\email{reza@cims.nyu.edu}

\address[Aukosh Jagannath]{Harvard University and University of Waterloo}
\email{a.jagannath@uwaterloo.ca}

\maketitle
\vspace{-.5cm}

{\begin{abstract}We study the algorithmic thresholds for principal component analysis of Gaussian $k$-tensors with a planted rank-one spike, via Langevin dynamics and gradient descent. In order to efficiently recover the spike from natural initializations, the signal to noise ratio must diverge in the dimension. Our proof shows that 
 the mechanism for the success/failure of recovery is the strength of the ``curvature" of the spike on the maximum entropy region of the initial data. To demonstrate this, we study the dynamics on a generalized family of high-dimensional landscapes with planted signals, containing the spiked tensor models as specific instances. We identify thresholds of signal-to-noise ratios above which order 1 time recovery succeeds; in the case of the spiked tensor model these match the thresholds conjectured for algorithms such as Approximate Message Passing. Below these thresholds, where the curvature of the signal on the maximal entropy region is weak, we show that recovery from certain natural initializations takes at least stretched exponential time. Our approach combines global regularity estimates for spin glasses with point-wise estimates, to study the recovery problem by a perturbative  approach. \end{abstract}} 

\section{Introduction}
Optimization in  high-dimensional landscapes can be computationally hard. 
This difficulty is often attributed to the topological complexity of the landscape. 
We show here that for planted signal recovery problems in high dimensions, there is another key obstruction to local optimization methods. Indeed, we find that a crucial factor in these settings is the competition between the strength of the signal and the entropy of the prior. We focus on a well-known {optimization} problem from high dimensional statistics which is known to be NP hard 
\cite{HiLi13}, namely {maximum likelihood estimation for} \emph{tensor principal component analysis} (PCA) \cite{MR14}. 

Suppose that we are given $M$ i.i.d.\ observations, $Y^\ell $, of a $k$-tensor of rank 1 which 
has been subject to Gaussian noise. That is, 
\[
Y^\ell = \sqrt N \lambda v^{\tensor k} + W^\ell\,,
\]
where $v\in \mathbb S^{N-1}(1)$ is deterministic, $W^\ell$ are i.i.d.\ Gaussian $k$-tensors with $W^{\ell}_{i_1,\ldots,i_k}\sim\cN(0,1)$
and $\lambda\geq0$ is the \emph{signal-to-noise ratio}. 
Our goal is to infer the ``planted signal'' {or ``spike"}, $v$, by maximum likelihood estimation. 

Observe that maximum likelihood estimation for this problem boils down to optimizing an empirical
risk of the form
\begin{equation}\label{eq:emp-risk-def}
\hat R(x) = \frac{1}{M}\sum_{\ell\leq M} (W^\ell,x^{\tensor k}) + \sqrt N \lambda (v,x)^k\,,
\end{equation}
where $(\cdot,\cdot)$ denotes the usual Euclidean inner product.
Note further that in this setting, optimizing this risk is equivalent (in law) to optimizing the same risk for a single observation
upon making the change $\lambda\mapsto \sqrt M \lambda$. {We therefore restrict our analysis to the case $M=1$.}

When $k=2$, this is the well-known spiked matrix model \cite{johnstone2000distribution}.
In this setting it is known \cite{Peche06} that there
is an order 1 critical signal-to-noise ratio, $\lambda_c$, such that below $\lambda_c$, it is information-theoretically impossible to detect the spike,
and above $\lambda_c$, the maximum likelihood estimator  is a distinguishing statistic. 
This transition is commonly referred to as the BBP transition~\cite{Baik2005phase}. 
In this setting the maximum likelihood estimator is the top eigenvector, which can be computed in polynomial time
by, e.g., power iteration. Much more detailed information is known about this transition for spiked matrix models, including universality, fluctuations, and large deviations. See, e.g.,~\cite{maida2007,capitaine2009,BGM12} for a small sample of these works.

When $k\geq3$, this is the spiked tensor model \cite{MR14}. 
In this case, there is a regime of signal-to-noise ratios for which it is
 information theoretically possible to recover the signal
 but for which there is no known algorithm to efficiently approximate it. This is called a \emph{statistical-to-algorithmic gap}. 
 In particular, it was shown in  \cite{MR14,mont15,perry2016statistical} that the minimal signal-to-noise ratio above which it is information-theoretically possible to detect the signal---called the \emph{information-theoretic threshold}---is of order~1. See also \cite{BMPW16,perry2016statistical, LMLKZ17,Ch17} for similar results with different priors.  
On the other hand, the minimal signal-to-noise ratio above which one can \emph{efficiently} detect the signal--- called the \emph{algorithmic threshold}---has been
proved or predicted to scale like $\lambda = N^{\alpha}$ for some $\alpha>0$ for every studied algorithm. 
(By the correspondence mentioned
above, the regime of diverging $\lambda$ can be translated to the regime of  $\lambda$ of order 1 with a diverging number of observations $M=N^{\alpha/2}$, so that this regime is also of practical interest.)
In \cite{MR14}, two local optimization methods, 
Approximate Message Passing and Tensor Power Iteration were shown to have critical exponents $\alpha$ at most $(k-1)/2$
with predicted thresholds at $\alpha=(k-2)/2$. Semi-definite relaxation approaches have also been analyzed. 
Tensor unfolding was shown~\cite{MR14} to have a critical exponent of at most $\frac{\lceil k/2\rceil-1}{2}$ and conjecturally $(k-2)/4$.  It was also shown that the degree $4$ Sum-of-Squares algorithm \cite{hopkins2015tensor} and a related spectral algorithm \cite{hopkins2016fast}  (in the case $k=3$)
have sharp critical thresholds of $(k-2)/4$. See also~\cite{kim2017community} for a similar analysis in the case $k=4$. 
We remark that statistical-to-algorithmic gaps, often diverging in the underlying dimension, have also been observed in myriad other problems of interest 
\cite{berthet2013optimal,zdeborova2016statistical,abbe2018proof,david2017high,barak2016nearly,barak2016noisy}.

Let us also discuss the complexity of the landscape $\hat R(x)$ given by~\eqref{eq:emp-risk-def}. The complexity in the absence of a spike (the case where $\lambda=0$) has been extensively studied~\cite{ABC13,ABA13,Sub15}; see also~\cite{GeMa17} for a related line of work. When adding in the signal term so that $\lambda>0$, it was proved~\cite{BMMN17} that
the expected number of critical points of $\hat R(x)$---called the annealed complexity---
is exponentially large in $N$ and has a topological phase transition as one varies $\lambda$ on the order 1 scale. 

One might wonder why the statistical-to-algorithmic gap is diverging when $k\geq3$. We investigate this issue
for algorithms which directly perform maximum likelihood estimation, by analyzing the behavior of
a family of ``plain vanilla''  algorithms, called \emph{Langevin dynamics}, as well as gradient descent. 
We find that for natural initializations, the statistical-to-algorithmic gap for Langevin dynamics and gradient descent diverges like $\lambda\sim N^\alpha$.
One may expect that this issue is due to the topological complexity of $\hat R(x)$. 
Our proof, however, suggests that this gap is actually due to the weakness of the signal in the region of maximal entropy
for the uninformative prior. 

To clarify this point, we study these dynamics on a more general family of random landscapes. 
For convenience, let us rescale our problem  to be on $\cS^N =\mathbb S^{N-1}(\sqrt N)$, the sphere in $\mathbb R^N$ of radius $\sqrt N$.  We consider a function $H:\cS^N\to \mathbb R$ of the form 
\begin{equation}\label{eq:ham-def}
H(x) = H_0 (x)-N\lambda \phi(x)\,,
\end{equation}
 where $\phi$ is a deterministic, non-linear function and $H_0$ is a noise term. 
 To put ourselves in a general setting, we only assume that   
$H_0$ is a mean-zero Gaussian process with a rotationally invariant law that is well-defined in all dimensions.
That is, for every $N$, $H_0$ has covariance of the form
\begin{equation}\label{eq:H-cov}
\cov({H_0(}x{)},{H_0(}y{)})=N\xi\left(\frac{(x,y)}{N}\right),
\end{equation}
for some fixed function $\xi$, where $(\cdot,\cdot)$ denotes the Euclidean inner product.\footnote{It is classical \cite{Schoen42} that the largest class of such $\xi$ is of the form
$\xi(t)=\sum_p a_p^2 t^p$ with $\xi(1+\eps)<\infty$ for some $\eps>0$.}
For simplicity, we  take the function $\phi(x)$ to  be a function of the inner product of $x$ with some ``unknown'' vector $v\in \mathbb R^N$. As $H_0$ is isotropic, without loss of generality, we assume that $v = e_1$, the first canonical Euclidean basis vector, so that $\phi(x)$ is a function of 
\begin{equation}
m_N (x) = \frac{{(x,e_1)}}{\sqrt N}= \frac{x_1}{\sqrt N}\,,
\end{equation}
which we call the \emph{correlation}.
In particular, we take $\phi$ of the form 
\begin{equation}
\phi(x)=  (m_N(x))^k\,,
\end{equation}
where $k\geq 1$ is not necessarily integer.  The case $\xi(t)=t^k$ and integer $k\geq 2$, corresponds to the setting of 
\eqref{eq:emp-risk-def}. The case where $\xi(t)=t^p$ corresponds to the $(p+k)$-spin glass  model from \cite{GiSh00},
whose topological phase transitions have been precisely analyzed in \cite{BBCR18} via a computation of the quenched complexity
using a novel replica-Kac-Rice approach. 

We  analyze here the performance of Langevin dynamics and gradient descent in achieving order~1 correlation as one varies the initialization,  the non-linearity of the signal, $k$, and the signal-to-noise ratio, $\lambda$. 
If $k>2$, we find that the critical threshold for algorithmic recovery via Langevin dynamics diverges like $\lambda_{alg}\sim N^{\alpha}$, 
with $\alpha>0$, for a natural class of  initializations. 
On the other hand,
we find that if $k<2$, this algorithmic threshold is of order 1. 
In the former regime, the second derivative of the signal is vanishing in the maximum entropy region of the uninformative 
prior, whereas in the latter it is diverging, matching the mechanism proposed above. 

Our analysis has two main thrusts: efficient recovery above critical thresholds and refutation below them.  
In both of these settings, we find that the obstacle to recovering the signal $e_1$ via Langevin dynamics is escaping the equator, i.e.,  the region where $m_N(x) = O(N^{-1/2})$, which corresponds to the maximum entropy region of the uninformative prior. 

For the recovery problem, we prove that as soon as $\alpha>\alpha_c(\infty):=(k-2)/2$, Langevin dynamics and gradient descent 
solve the recovery problem in order 1 time when started uniformly at random: see Theorem~\ref{thm:vol-recovery}.  To isolate the importance of the initialization in this problem,
we provide a hierarchy of sufficient conditions on the initial data that imply 
that Langevin dynamics with $\lambda=N^{\alpha}$ will efficiently solve the recovery problem down to a hierarchy of thresholds $\alpha_c(n)$, 
the lowest of these thresholds being $\alpha_c(\infty)$: see Section~\ref{sec:general-recovery}.
In Section~\ref{sec:examples}, we give examples of initial data that satisfy these conditions at different levels: the case of the  volume  measure is discussed in depth in Section~\ref{sec:vol-recovery}.  To prove these results, we build on the ``bounding flows'' strategy of \cite{BGJ18a}. In particular, we show that
on $O(1)$ times, we can compare the evolution of the correlation, $m(X_t)$, to the gradient descent 
for the problem with no noise $H_0=0$. This follows by a stochastic Taylor expansion upon combining the Sobolev-type $\cG$-norm
estimates developed in \cite{BGJ18a} for spin glasses, with estimates on the regularity of the initial data developed here.
This is discussed in more detail in Section~\ref{sec:ideas}.

For the refutation problem, the threshold $\alpha_c(\infty)$ also has a natural (heuristic) interpretation: if $\alpha<\alpha_c(\infty)$, 
then Langevin dynamics started uniformly at random would take exponential time to solve
the recovery problem when given a pure signal, i.e., $H_0(x)=0$. We conjecture that this is also the case when $H_0$ is added back, so that $\alpha_c(\infty)$ is sharp for uniform at random initialization.
As added motivation for this conjecture, we prove that indeed $\alpha_c(\infty)$ is sharp for a natural Gibbs class of initial data: from such initializations,  the dynamics efficiently recover
the signal for $\alpha>\alpha_c(\infty)$, but take at least stretched exponential time to do so for $\alpha<\alpha_c(\infty)$.
To prove this refutation theorem below $\alpha_c(\infty)$, we formalize the notion of \emph{free energy wells} 
(see  Definition~\ref{def:GFEB}), whose existence implies exponential lower bounds on the exit time of 
the well from a restriction of the corresponding Gibbs measure.   
We find that below the critical $\alpha_c(\infty)$, there is a free energy well around the equator, 
and use this to deduce hardness of recovery for Gibbs initializations.  For more on this, see Section~\ref{sec:refutation}.

\section{Statements of main results}\label{sec:main-results}
We focus on a canonical class of optimization algorithms
called \emph{Langevin dynamics} and \emph{gradient descent} with \emph{Hamiltonian} $H$.
The Langevin dynamics interpolate between Brownian motion and gradient descent via a parameter $\beta>0$,
usually called the \emph{inverse temperature}, with the case $\beta=0$ corresponding to Brownian motion and the case $\beta = \infty$ corresponding to gradient descent. 
More precisely, for $0\leq \beta<\infty$, let $X^\beta_t$ solve the stochastic 
differential equation (SDE)
\begin{equation}\label{eq:X-def}
\begin{cases}
dX^\beta_t &= \sqrt{2}dB_t -\beta \nabla H(X^\beta_t) dt\\
X^\beta_0 &= x
\end{cases}\,,
\end{equation}
where $B_t$ is Brownian motion on $\cS^N$, $\nabla$ denotes the covariant derivative on $\cS^N$, and $H$ is called the \emph{Hamiltonian} which is given here by \eqref{eq:ham-def}. 
 As $H$ is  $C^{1}$, 
this martingale problem is well-posed so that $X^\beta_t$ is well-defined \cite{EthierKurtz86,StroockVaradhan06}. (When $k$ is 
an integer, $H$ is smooth so that one can solve this in the strong sense as well \cite{Hsu02}.)  When $\beta=\infty$, let $X^\infty_t$ denote the solution
to the ODE
\begin{equation}\label{eq:GD-def}
\begin{cases}
dX^\infty_t &= -\nabla H(X^\infty_t) dt\\
X^\infty_0 &= x
\end{cases}\,.
\end{equation}
Note that on the complement of the set $m_N(x)=0$, $\nabla H$ is locally Lipschitz so that $X_t^\infty$ is locally well-posed (in time) by the Picard-Lindelof theorem.
Note that the infinitesimal generators of these processes are given by
\begin{equation}\label{eq:L-def}
L_\beta= \begin{cases} \Delta -\beta \g{\nabla H,\nabla\cdot} & \quad \beta \in (0,\infty) \\ -\g{\nabla H,\nabla\cdot } & \quad \beta = \infty\end{cases}\,.
\end{equation}
where $\Delta$ is the Laplacian on $\cS^N$, {and $\g{\cdot,\cdot}$ is the metric tensor}. Whenever $\beta$ is clear from context, we will write $X_t =X_t^\beta$ and $L=L_\beta$. 


We aim to determine the minimal $\lambda$ for which efficient recovery of the signal, $ e_1$, is possible via these dynamics and understand the role that the initialization plays. There are, of course, multiple notions of recovery. The main ones in which we are interested
are \emph{weak recovery} and \emph{strong recovery}.
For fixed $\xi, \beta, k$, sequence $\lambda_N$, and sequence of initial data $x_N$, we say that the Langevin dynamics \emph{weakly recovers} the signal in order $1$ time if it attains order 1 correlation in $O(1)$ time with high probability. 
On the other hand, we say that the dynamics \emph{strongly recovers} the signal in order $1$ time if it attains $1-o(1)$ correlation in $O(1)$ time with high probability. 
In the diverging signal-to-noise ratio
regime, weak and strong recovery are equivalent (see Lemma~\ref{lem:weak-to-strong}).  A more complete discussion of their relationship, is provided in Section~\ref{sec:performance-diverging}.

\subsection{Recovery initialized from the volume measure}\label{sec:vol-recovery}
Perhaps the most natural initialization is a completely uninformative prior, i.e., the (uniform) volume measure on $\cS^N$. This is particularly motivated from the algorithmic perspective as it is easy to sample from the volume measure on $\cS^N$ in order 1 time (the volume measure has a log-Sobolev inequality with constant uniformly bounded away from $0$~\cite{Led01}, so that running Brownian motion serves as a fast sampler from the volume measure). In order to focus on the key issues and deal with all $k$ in a comprehensive manner, we restrict to the upper hemisphere: $\{x_1 > 0\}$. Of course, a point sampled from the volume measure on $\cS^N$ is in the upper hemisphere with probability $\frac 12$. 

We obtain the following recovery guarantees starting from the volume measure on the upper hemisphere.\footnote{In the first preprint of this article, a suboptimal version of this result for initialization from the uniform measure was stated; the current formulation of Theorem~\ref{thm:vol-recovery} resolves what was there labeled as Conjecture 1.} Let $dx$ be the volume measure on $\cS^N$ and let $\mathbb P$ denote the law of the noise $H_0$. In the following, denote the law of $X^\beta_t$ started at $x$ by $Q_x$, and in the case of gradient descent, referred to as $\beta = \infty$, interpret $Q_x$ as a Dirac mass on the trajectory of $X^\infty_t$ started from $x$. 

\begin{theorem}\label{thm:vol-recovery}
Fix any $\xi$, any $\beta\in(0,\infty)\cup\{\infty\}$, and any $k\in [1,\infty)$. 
\begin{enumerate}
\item If $k>2$ and $\lambda = N^{\alpha}$ for $\alpha>\frac{k-2}{2}$, then for every $\epsilon>0$, there exists $T_0$ such that for every~$T\geq T_0$, 
\begin{align*}
\lim_{N\to\infty} \int_{x_1 > 0} Q_{x}\Big(\min_{t\in [T_0,T]} m_N(X^\beta_t) \geq 1-\epsilon\Big) dx =1\,, \qquad \mbox{$\mathbb P$-a.s.}
\end{align*} 
\item If $k<2$, $\beta<\infty$, and $\lambda$ is a large enough constant, there exists $T_0$ such that for all $T\geq T_0$,
\begin{align*}
\lim_{\epsilon \to 0} \liminf_{N\to\infty} \int_{x_1 > 0} Q_{x}\Big(\min_{t\in [T_0,T]} m_N(X^\beta_t) \geq \epsilon\Big) dx =1\,, \qquad \mbox{$\mathbb P$-a.s.}
\end{align*} 
\end{enumerate} 
\end{theorem}

We first pause to comment on the special case of $\xi(t)=t^k$ with $k$ integer, corresponding to maximum likelihood estimation for tensor PCA. The thresholds of Theorem~\ref{thm:vol-recovery} improve upon the rigorously known threshold for Approximate Message Passing and Tensor Power Iteration, and the $({k-2})/{2}$ threshold matches the conjectured threshold for those algorithms. 

These thresholds correspond to the signal-to-noise ratios for which the second derivative of $N \lambda \phi$ diverges at points of correlation $m_N(x)=\Theta (N^{-\frac 12})$ (the asymptotic support of $dx$). 
We predict that these thresholds are sharp for efficient algorithmic recovery of $e_1$ via local optimization, so that when the second derivative is $o(1)$ at these correlations, efficient recovery is not possible. When $H_0 \equiv 0$, it is easy to see that for $k<2$ and $\lambda$ order 1, or $k>2$ and $\lambda = N^{\alpha}$ with $\alpha<\frac {k-2}{2}$, the Langevin dynamics takes an exponential time to correlate with $e_1$ with this initialization; it would be expected that this persists with the addition of noise: see the in depth discussion in Section~\ref{sec:ideas}.

We are able to prove sharpness of these thresholds for a high-temperature Gibbs-type initialization that approximates the volume measure as $\beta \to 0$, i.e., that below these thresholds, the corresponding dynamics fails to recover in stretched exponential time whereas above the thresholds the dynamics recover as above. In the next section, we define general conditions, \CI~ and \CII, on the initial data that guarantee recovery above these thresholds.

\subsection{General thresholds for recovery}\label{sec:general-recovery}
Theorem~\ref{thm:vol-recovery} is a particular case of more general recovery results that translate a natural hierarchy of conditions on the initial data to thresholds in $\lambda$ above which we guarantee recovery of $e_1$.  
Let $\cM_1(\cS^N)$ denote the space of probability measures on $\cS^N$. A choice of initial data corresponds to a choice of measure  $\mu_N \in \cM_1 (\cS^N)$. 
Our main recovery guarantees apply to any initial data which satisfy the following two natural conditions. 

The first condition is on the regularity of the initial data. 
For $\beta \in (0,\infty)\cup \{\infty\}$, define $L_0$ as 
\begin{equation}\label{eq:L_0-def}
L_{0,\beta}=
\begin{cases}\Delta -\beta\g{\nabla H_0,\nabla\cdot} & \quad \beta\in (0,\infty) \\ 
-\langle \nabla H_0 ,\nabla \cdot \rangle & \quad \beta = \infty
 \end{cases}\,,
\end{equation}
i.e., the generator of Langevin dynamics and gradient descent with respect to $H_0$. 
Correspondingly, let  $e^{tL_{0,\beta}}$ refer to the semigroup induced by $L_{0,\beta}$.
For every $\delta>0$, $n\geq 1$, and $N$, let
\[
E_{n,\delta,N}^\beta=\bigcap_{\ell=0}^{n-1}\left\{x: \abs{L_{0,\beta}^{\ell}m_N(x)}\leq N^{-\frac{1}{2}+\delta}\right\}\,,
\]
and for every $\delta,T>0$, and $N$, let
\[
\tilde{E}_{T,\delta,N}^\beta=\left\{x: \sup_{t\leq T} \abs{e^{t L_{0,\beta}}L_{0,\beta} m_N(x)}\leq N^{-\frac{1}{2}+\delta}\right\}\,.
\]
Again, when understood from context, we drop the dependence on $\beta$ in the notation.

\begin{definition}
We say that a sequence of random probability measures $\mu_{N}\in\cM_1(\cS^{N})$
satisfies \textbf{\CI~ at level n} for inverse temperature $\beta\in (0,\infty)\cup \{\infty\}$ if for every $\delta>0$, 
\begin{equation}\label{eq:level-n}
\lim_{N\to\infty} \mu_{N}\left((E_{n,\delta,N}^\beta)^{c}\right)=0\quad \mbox{$\mathbb P$-a.s.}
\end{equation}
We say that a sequence of random probability measures $\mu_N\in\cM_1(\cS^N)$ \textbf{weakly}
satisfies \textbf{\CI~at level $\infty$} for inverse temperature $\beta\in (0,\infty)\cup \{\infty\}$ if for every $\delta,T>0$, 
\begin{equation}\label{eq:weak-level-infty}
\lim_{N\to\infty}\mu_N\left((\tilde{E}_{T,\delta,N}^\beta)^c\right)=0\quad \mbox{$\mathbb P$-a.s.}
\end{equation}
\end{definition}

Let us pause for a moment to interpret \CI~ intuitively. Let us call Langevin dynamics or gradient descent with Hamiltonian $H_0$ the ``pure noise" dynamics. Recall that $e^{tL_0} f(x)$ is the expected value
of $f$ with respect to the pure noise dynamics at time $t$ started at $x$. The set $\tilde{E}_{T,\delta,N}$ is then the set of initializations for which the pure noise dynamics does not have an atypically strong push in the (opposite) direction of the spike before time $T$. Since the landscape is isotropic, one should expect that if one runs the pure noise dynamics started from any initial data that is agnostic to the spike, then correlations of observables with the spike will stay on their central limit theorem scales for a long time. Weakly satisfying \CI~ at level $\infty$ states that the probability of such initializations with respect to the initial data tends to 1. On the other hand, \CI~ at level $n$ says that the first  $n$ coefficients in the Taylor expansion of $e^{tL_{0}} m$ are all on this typical scale. Clearly, $E_{\infty,\delta,N}$ is contained in $\tilde E_{T,\delta',N}$ for every $\delta'>\delta$, hence the qualifier ``weak". The implication in the other direction does not obviously hold due to possible cancellations.   

The second condition ensures that the initial correlation
is on the typical $\Theta(N^{-\frac 12})$ scale, so that the drift from gradient descent for the signal is not negligible at time zero. 

\begin{definition}
A sequence of random probability measures, $\mu_{N}\in \cM_1 (\mathcal{S}^{N})$, satisfies \CII, if
\begin{align*}
\lim_{\epsilon\to0}\varlimsup_{N\to\infty}\mu_{N}(x_{1}<\epsilon) & =0\, \quad \mbox{$\mathbb P$-a.s.}
\end{align*}
 A sequence of random probability measures $\mu_{N}\in \cM_1(\mathcal{S}^{N})$ satisfies \CII' if for every $\delta>0$,
\begin{align*}
\lim_{N\to\infty}\mu_{N}(x_{1}\leq N^{-\delta}) & =0\, \quad \mbox{in prob.}
\end{align*}
\end{definition}
We emphasize that neither of these conditions involve the parameters $k$ (the non-linearity of the signal) or $\lambda$ (the signal-to-noise ratio). The conditions can be shown to hold for various natural choices of initial data,
such as the volume measure on the upper hemisphere, implying Theorem~\ref{thm:vol-recovery}, as well as certain ``high-temperature''
Gibbs measures. For more on this, see~\prettyref{sec:examples}.

Let us now turn to our main results under \CI~ and \CII.
We begin with the \emph{supercritical} regime, $k>2$, where one will need $\lambda$ to diverge with $N$ 
 to efficiently recover, as the curvature of the signal in the region where $m_N(x)= \Theta(N^{-1/2})$ is negligible. 
For every $n\geq 1$, let
\begin{equation}
\alpha_c(n)=\frac{k-1}{2}-\frac{n-1}{2 n}\qquad \mbox{and}\qquad \alpha_c(\infty)=\frac{k-2}{2}\,.
\end{equation}
We then have the following result regarding strong recovery.

\begin{theorem}\label{thm:supercritical-main}
Fix any $\xi$, any $\beta\in (0,\infty)\cup \{\infty\}$, and  any $k>2$. Let $\lambda = N^\alpha$ and consider a sequence of initializations $\mu_N\in \cM_1(\cS^N)$. 
If $\mu_N$ satisfies \CII~, then we have the following.
\begin{enumerate}
\item If   $\mu_N$ weakly satisfies \CI~ at level $\infty$ for inverse temperature $\beta$, then  for every $\alpha>\alpha_c(\infty)$ and every $\eps>0$, there exists a $T_0$ such that for every $T>T_0$,
\begin{align*}
\lim_{N\to\infty} \int Q_x \Big(\min_{t\in [T_0,T]} m_N(X^\beta_t) \geq 1-\eps\Big)d\mu_N(x)=1 \quad \mbox{$\mathbb P$-a.s.}
\end{align*}
\item If  $\mu_N$ satisfies \CI~ at level $n$ for inverse temperature $\beta$, then the above holds for every $\alpha>\alpha_c(n)$. 
\end{enumerate}
If instead $\mu_N$ satisfies \CII', then the above convergence holds in probability.
\end{theorem}

The above theorem shows that in the regime $k>2$, we need $\lambda$ to diverge for Langevin 
dynamics and gradient descent to recover the signal in order 1 time. Observe that for such $k$, the second derivative of $N \phi$ in the region $m_N(x)=O(N^{-1/2})$
is vanishing as $N\to \infty$. 
Let us now show conversely, that in the \emph{subcritical} regime $k<2$,
i.e., the regime where the second derivative of $N\phi$ is diverging when $m_N(x)=\Theta(N^{-1/2})$, order 1 time weak recovery holds for large but finite signal-to-noise ratios. That is, the statistical-to-algorithmic gap (if one exists) is at most order 1 for $k<2$.
In this regime, one cannot hope for strong recovery (see \prettyref{rem:cant-hope}). 
If we let 
\begin{equation}
k_c(n)=2-\frac {1}{n}\qquad \mbox{and}\qquad k_c(\infty) = 2\,.
\end{equation}
then we have the following weak recovery guarantee. In the following, in the case $\beta=\infty$ and $1<k<2$ we interpret $Q_x$ to be 
the zero measure (i.e., $Q_x =0$) when $x_1 = 0$ to avoid issues related to well-posedness. For more on this, see \prettyref{rem:k<2}.

\begin{theorem}\label{thm:subcritical-main}
Fix any $\xi$ and any $\beta\in(0,\infty)\cup \{\infty\}$. There exists $\lambda_0(\beta,\xi,k)>0$ such that for all $\lambda>\lambda_0$ the following holds. 
If $\mu_N$ satisfies \CII~, then we have the following.
\begin{enumerate}
\item If $\mu_N$  weakly satisfies \CI~at level $\infty$ for inverse temperature $\beta$, then for every $1\leq k < k_c(\infty)$ 
and every $\eta>0$,  there exists $\epsilon>0$ and $T_0>0$, such that for any $T>T_0$
 \begin{equation}\label{eq:subcritical-main}
\liminf_{N\to\infty} \int Q_x \Big(\min_{t\in[T_0,T]} m_N(X^\beta_t) \geq \eps\Big)d\mu_N(x)>1-\eta \quad \mbox{$\mathbb P$-a.s.}
\end{equation}
\item If instead, $\mu_N\in \cM_1(\cS^N)$  satisfies \CI~at level $n$ for inverse temperature $\beta$, then the same result holds for every $1\leq k<k_c(n)$.
\end{enumerate}
\end{theorem}
The main ideas behind Theorems~\ref{thm:supercritical-main}~and~\ref{thm:subcritical-main} are essentially the same. We will explain the intuition behind their proofs presently
(see  \prettyref{sec:ideas} below). We end this section with the following remark  on possible relaxations of \CI. 
\begin{remark}
One could not decrease the sets $E_{n,\delta,N}$ and $\tilde E_{T,\delta,N}$, since measures that don't contain information about the planted signal, e.g., the volume measure on $\cS^N$, have $m_N(x)=\Theta(N^{-1/2})$.
\end{remark}

\subsection{Examples of initial data satisfying \CI~and \CII}\label{sec:examples}
Let us now turn to some examples of initial data that satisfy the conditions of our theorems.
When considering initial data for such problems there are a few natural choices.

\begin{example}
Let us begin by observing that $\mathbb P$-a.s., any initial data which is concentrated on the region $m_N(x)=\Theta(N^{-1/2})$,
e.g., $\delta_x$ for any $x$ having $m_N(x)=\Theta(N^{-1/2})$,
satisfies \CI~ at level 1 for every $\beta \in (0,\infty)\cup \{\infty\}$ and \CII~tautologically. 
\end{example}

\begin{example}
Initialization from the volume measure was discussed at length in Section~\ref{sec:vol-recovery}. Theorem~\ref{thm:vol-recovery} is an immediate corollary of Theorems~\ref{thm:supercritical-main}--\ref{thm:subcritical-main} combined with the following theorem.

\begin{theorem}\label{thm:condition-check-volume}
The normalized volume measure  on $\cS^N\cap\{x_1> 0\}$ weakly satisfies
\CI~at level~$\infty$ at every $\beta \in (0,\infty)\cup \{\infty\}$ and satisfies \CII.
\end{theorem}

At an intuitive level, Theorem~\ref{thm:condition-check-volume} is saying that for the Langevin dynamics or gradient descent with respect to the pure-noise environment $H_0$, initialized from the normalized volume measure, the observable $L_0 m$ does not leave its typical scale of $O(N^{-\frac 12})$ in order one timescales: this is very natural to expect as $L_0m$ consists of the correlation with $e_1$, and the gradient of $H_0$ in the direction of $e_1$, while the volume measure and $H_0$ are both isotropic and have no sense of the direction $e_1$. While a similar intuition should hold for a broad set of initializations, the proof of Theorem~\ref{thm:condition-check-volume}, found in Section~\ref{sec:initial-data-volume}, relies crucially on the rotational invariance of the volume measure.\footnote{In a previous version of this paper, the sub-optimal version of Theorem~\ref{thm:vol-recovery} followed from checking \CI~at level 4 by hand in the absence of weak \CI~ at level $\infty$.}
\end{example}

\begin{example}
Another natural class of initializations are Gibbs type initializations in the ``pure noise" environment.
Let $d\pi_{0,\beta}(x)\propto \exp (-\beta H_0(x))dx$ be the Gibbs measure on $\cS^N$ corresponding only to the noise $H_0$ at inverse temperature $\beta$, and let $\pi_{0,\beta}^+$ be 
$\pi_{0,\beta}$ conditioned on $\{x_1> 0\}$, i.e., 
\begin{equation}\label{eq:pi_0-def}
\pi_{0,\beta}^+ = \pi_{0,\beta}(\cdot \mid x_1> 0).
\end{equation} 
One can show that such measure satisfy \CI~ at level $n$ for every $n$.
\begin{theorem}
Let $\xi$ be even. There exists $\beta_0>0$ such that for all $\beta<\beta_0$, the measure $\pi_{0,\beta}^+$ satisfies \CI~ at level $n$ for $\beta$, for every $n\geq 1$. Moreover, $\pi_{0,\beta}^{+}$ satisfies \CIIprime.
\end{theorem}

\begin{corollary}\label{cor:recovery-spin-glass}
Let $\xi$ be even and $k>2$. If $\beta<\beta_0$, for every $\beta>0$, if $\lambda = N^\alpha$ with $\alpha>\alpha_c(\infty)$, the Langevin dynamics starting from $\pi_{0,\beta}^{+}$ strongly recovers the signal in order 1 time in $\mathbb P$-prob.
\end{corollary}

Notice, in particular, that like the volume measure, the measure $\pi_{0,\beta}$ is completely independent of the noise $\phi$ as well as the signal-to-noise ratio $\lambda$. Moreover,  as $\beta \to 0$, the measure $\pi_{0,\beta}$ approximates the volume measure.
We end this section with the following conjecture regarding the measure $\pi_{0,\beta}$.
This result would imply an almost sure recovery result for $k>2$ and $\alpha>\alpha_c(\infty)$, as well as the matching weak recovery result for $k<2$.
We also believe that it is of independent interest in the statistical physics community.

\begin{conjecture*}
For every $\beta>0$, if $\gamma(\cdot)$ is the law of a standard Gaussian,
\[
\pi_{0,\beta}(x_1\in \cdot) \xrightarrow[N\to\infty]{} \gamma(\cdot)
\]
weakly as measures $\prob$-a.s. In particular, $\pi_{0,\beta}^+$ satisfies \CII.
\end{conjecture*}
\end{example}

\subsection{Refutation below $\alpha_c(\infty)$}\label{sec:refutation}
The Gibbs initialization introduced in the previous section, $\pi_{0,\beta}^+$, is of further importance to us as it is an initialization for which we are able to prove the sharpness of the threshold $\alpha_c(\infty)$. 
For any $\eps>0$, let $\tau_\eps$ be the hitting time
\[
\tau_\eps = \inf\{t>0: m_N(X_t)\geq \eps\},
\]
and let $\pi_{0,\beta}^+$ be as in \eqref{eq:pi_0-def}.

\begin{theorem}\label{thm:worst-case-main}
Fix any $\xi$, any $\beta\in (0,\infty)$ and $k>2$, and let $\lambda_N=  N^{\alpha}$.
If $\alpha<\alpha_{c}(\infty)$, there exists $c>0$ such that for every $\eps>0$ sufficiently small,
\[
\limsup_{N\to\infty} \frac{1}{N^{2\eps}}\log\int Q_x\Big(\tau_{2N^{-\frac 12 +\epsilon}}\leq e^{cN^\eps}\Big)d\pi_{0,\beta}^{+}(x)< -c\,, \qquad \mbox{$\prob$-a.s.}
\]
\end{theorem}
Motivated by the fact that in the limit as $\beta\to 0$, $\pi_{0,\beta}^+$ approximates the volume measure on $\cS^N\cap \{x_1> 0\}$, we believe that a similar refutation result also holds for initialization from the volume measure whenever $\alpha<\alpha_c(\infty)$; this would make the thresholds $\alpha_c(\infty)$ and $k_c(\infty)$ sharp for initialization from the volume measure. 

\subsection{Ideas of proofs } \label{sec:ideas}
We now sketch some of the key ideas underlying the above recovery and refutation results and their proofs.
\subsubsection{Ideas of proofs of \prettyref{thm:supercritical-main} and \prettyref{thm:subcritical-main}}\label{sec:ideas-recovery}
We first discuss the intuition behind the proofs of Theorems~\ref{thm:supercritical-main}--\ref{thm:subcritical-main}. 
Our interest is in understanding
the transition for signal recovery in short times. It turns out that 
the subcritical and supercritical problems are essentially the same. To see why,
consider, for the moment, the recovery question for Langevin dynamics in the simpler setting
where there is only a signal, $H= -N \lambda\phi$ and $H_0 \equiv 0$.

By rotation invariance, the question of escaping the equator for the problem with pure signal 
is effectively the same as studying
the escape from the origin for a 1-dimensional Langevin dynamics with
Hamiltonian, 
\begin{equation}\label{eq:V(m)}
V(m)=\beta \lambda  m^k +\frac{1}{2}\log(1-m^2)\approx \beta \lambda m^{k}-\frac{1}{2}m^{2}\,,
\end{equation}
in the small noise regime (noise of order $N^{-1/2}$). Evidently, this amounts to studying
the ODE,
\begin{equation}\label{eq:m-dot-pure-spike}
\dot{m}=\beta \lambda  km^{k-1}-m\,,
\end{equation}
where $m(t)=m_N(X_t)$. The second term reverts to the origin.
To escape, we must then hope that the first term
dominates at the initial point. 
In particular, if $m$ is positive
and small and $\lambda$ is large, one hopes to compare this ODE to
the simpler system. 
\begin{equation}\label{eq:m-dot-grad-descent}
\dot{m}\approx \beta \lambda  km^{k-1}\,.
\end{equation}
In this setting, one may then apply a standard comparison inequality (see \prettyref{lem:power-law-comparison}), 
which compares solutions of this ODE to certain power laws.\footnote{In the critical regime, $k=2$, this is the classical Gronwall inequality which is, of course, not of power law type.}
Evidently, the order of growth of the second derivative of $m^{k}$ is the essential ingredient in resolving the tradeoff above. Indeed, under \CII, which places $m_N(0)= \Theta(\frac{1}{\sqrt N})$, if $k<2$ and $\lambda$ is order 1, then $m^{k-1}\gg m$, and similarly, if $k>2$ then $\lambda m^{k-1}\gg m$ provided $\lambda$ grows sufficiently fast ($\alpha>\alpha_c(\infty)$). Notice also that if $\alpha<\alpha_c(\infty)$, then Langevin dynamics with $\beta<\infty$ from $m_N(0)=O(\frac{1}{\sqrt N})$ would not efficiently recover the signal even in this trivial pure spike problem.

When adding back $H_0$, we consider the evolution equation for $m$ given by
\begin{align}\label{eq:m-evolution-Langevin}
dm = \left(\beta\lambda k m^{k-1}(1-m^2)+L_0 m\right)dt+dM^m_{t}\,,
\end{align}
where $L_0$, given by \eqref{eq:L_0-def}, is the infinitesimal generator for Langevin dynamics
with respect to $H_0$, and $M_t^m$ is a martingale.
We will see that $M_{t}^m=O(\sqrt{t/N})$,  so that on short times, this is not far from the situation of~\eqref{eq:m-dot-grad-descent}.
The remaining discrepancy, evidently, 
is to ensure that $|L_0 m| $ starts and remains smaller than $\lambda m^{k-1}$.
To this end, we use the $\mathcal G$-norm estimates from~\cite{BGJ18a} to show that provided $L_0m(X_0)$ is suitably localized,
i.e., provided \CI~ holds at level $n$, then $L_0m(X_t)$ remains localized
on the relevant timescale needed to recover the signal above $\alpha_c(n)$ (see Theorem~\ref{thm:taylor-bound}). 
The main result  then follows by combining this localization, or its weaker version at level~$\infty$, with the 
comparison inequality of Lemma~\ref{lem:power-law-comparison}: this is developed in Section~\ref{sec:randomized-recovery}.

The proof of the recovery result for $\beta = \infty$ follows mutatis mutandis as that of $\beta<\infty$. 
For an explanation of the relevant modifications, see Section~\ref{subsec:gd-proof}. 

\begin{remark}(Strong recovery is impossible for finite $\lambda$)\label{rem:cant-hope} 
When $\lambda$ is order 1, one cannot hope to obtain a strong recovery result. Indeed if we start from any point sufficiently close to the north pole, correlation $m_N(x)\geq 1-\epsilon$ for $\epsilon>0$ sufficiently small, then $m_N(t)$ will decrease in correlation in order 1 time. 
To see this, we examine the drift in~\eqref{eq:m-evolution-Langevin}, and expand $L_0$ as 
$$L_0 m= -m-\beta\langle \nabla H_0, \nabla m\rangle\,:
$$ if $\epsilon$ is sufficiently small, then $-m$ will dominate $\beta \lambda km^{k-1}$; furthermore the maximum of $|\langle \nabla H_0,\nabla m\rangle|$ on the spherical cap $\{m_N(x)\geq 1-\epsilon\}$ can be shown to scale down to zero as $\epsilon$ goes to zero. Putting these together with the types of arguments found in Section~\ref{sec:performance-diverging} would imply the desired. 
\end{remark}

\subsubsection{Ideas of proofs of \prettyref{thm:worst-case-main}}
The underlying idea behind our refutation result is the presence of what we call  a \emph{free energy well}
for the correlation, which is defined as follows. Define the Gibbs measure for $H$ by, 
\[
d\pi_{\beta}(x) \propto \exp(-\beta H(x))dx\,,
\]
which is normalized to be a probability measure, where $dx$ is the normalized volume measure on $\cS^N$. In the following, for a real number $a$, we let $B_\eps(a)=\{x:\abs{x-a}<\eps\}$, denote the ball of radius $\eps$ around $a$.
For any function $f:\cS^N \to \mathbb R$, we define the entropy: 
\begin{align*}
I_f(a;\epsilon)= -\log \pi(\{x: f(x)\in B_\epsilon (a)\})\,.
\end{align*}
We can now define free energy wells for Lipschitz functions.
\begin{definition}\label{def:GFEB}
A Lipschitz function, $f:\cS^N\to\R$, has an $\epsilon$--\emph{free energy well} of height $h$ in $[a,b]$ if the following holds: there exists $c\in (a,b)$ and $\eta>0$ such that $B_\epsilon (a)\cap B_\epsilon (b)\cap B_\eta(c)=\emptyset$ and 
\begin{align*}
 \min\{I_f (a,\epsilon),I_f (b,\epsilon)\} -  I_f (c,\eta) \geq h \,.
\end{align*}
\end{definition}
Such free energy wells are the exit time analog of the of free energy barriers 
formalized in \cite{BAJag17} for spectral gap estimates. We show in \prettyref{thm:fe-barrier-metastability} that free energy wells confine the dynamics on timescales that are exponential in the height, $h$, when started from this Gibbs measure $\pi_\beta$ restricted to the well. 
We then show that for $\alpha<\alpha_c(\infty)$, there is a free energy well for the correlation: namely, the function $f(x)=(x,e_1)$ has a free energy well of height $N^\epsilon$ in $[-N^{\epsilon},N^{\epsilon}]$ (see Proposition~\ref{prop:x_1-GFEB-2}).
Theorem~\ref{thm:worst-case-main} then follows by combining this with 
the facts that  $\pi_{\beta}$ and $\pi_{0,\beta}$ are comparable when restricted to this band of correlations, and $\pi_{0,\beta}$ is asymptotically supported in this region.

We conclude with a remark regarding exceptional points which facilitate recovery at order one~$\lambda$.

\begin{remark}[Equatorial passes]\label{rem:eq-pass}
In light of the free energy well for the correlation, one might hope to prove an even stronger 
refutation theorem. It may be tempting to believe that when $k>2$ and $\lambda$ is order 1, the Langevin dynamics cannot recover the signal in sub-exponential times, uniformly over \emph{all} initial $X_0$ with $m_N(X_0) = O(N^{-\frac 12})$. Indeed, this is the case for the simpler ``pure signal'' problem where $H_0=0$. As a consequence of~\eqref{eq:m-evolution-Langevin} and lower bounds on $\|\partial_1 H_0\|_\infty$ (see e.g.,~\cite{ABA13}), however, this guess does not hold. One can show the following: for every $\xi,k$ and $\lambda>\lambda_0(\xi,k)$, there exist initial data $X_0$ such that $m_N(X_0) = O(N^{-\frac 12})$ and such that Langevin dynamics started from $X_0$ succeeds at weak recovery. \end{remark}

\subsection*{Acknowledgements}
We thank Giulio Biroli, Chiara Cammarota, Florent Krzakala, Lenka Zdeborova, Afonso Bandeira, and David Gamarnik for inspiring discussions. 
This research was conducted while G.B.A.\ was supported by BSF 2014019 and NSF DMS1209165, and A.J.\ was supported by NSF OISE-1604232. R.G.\ thanks NYU Shanghai for its hospitality during the time some of this work was completed. 

\section{Preliminaries: Regularity theory and stochastic analysis in high dimensions} 
Throughout the paper, we will make frequent use of certain uniform Sobolev-type estimates
for $H_0$, developed in the context of spin-glass dynamics, as well as properties of solutions to certain Langevin-type stochastic differential equations.
We recall these results in this section. In what follows, for functions $f,g$ we say that $f\lesssim g$
if there is a constant $C>0$ such that $f\leq C g$.

\subsection{Regularity theory of spin glasses and the  $\mathcal G-$norm} 
As is often the case in such problems it will be important to understand the regularity
of the related Hamiltonians. It turns out that in high-dimensional analysis problems,
one needs to define these Sobolev spaces carefully, as the scaling
of the norms in the dimension is often crucial to the problem at hand. 

With this in mind, let us recall the  $\cG^k$-norm, which provides a (topologically) equivalent norm
on the usual Sobolev space, $W^{k,\infty}$, but which is better suited to high-dimensional problems,
as well as the related $\cG$-norm regularity of $H_0$, established 
in~\cite{BGJ18a}, which will be crucial to our analysis.

\begin{definition}
A function $f:\cS^N \to \mathbb R$ is in the space $\cG_{K}^k(\cS^N)$, if 
\[
\norm{f}_{\cG^k_K} := \sum_{0\leq \ell\leq k} K^{{\ell}/2}\norm{\,|{\nabla^{\ell} f}|_{op}\,}_{L^\infty(\mathcal S^N)} <\infty\,.
\]
Here, $\abs{\nabla^k f}_{op}(x)$ denotes the  natural operator norm when $\nabla^k f$ is viewed as
a $k$-form acting on the $k$-fold product of the tangent space $T_x \cS^N$. Throughout the paper, unless otherwise specified $|\nabla^k f|$ will denote this norm. 
\end{definition}
\begin{remark}
By the equivalence of norms on finite dimensional vector spaces, 
for fixed $N$, this space is equivalent to the canonical sobolev space $W^{k,\infty}$,
which is defined using Frobenius norms. 
We use the operator norm instead for the following reason. 
For $n\geq2$, we need to bound the operator norms
of random tensors. It is well-known that
for such random tensors, there is a marked difference in
the scaling of the Frobenius and operator norms in the dimension (see, e.g.,~\cite{Vershynin}). 
\end{remark}

We let $\mathcal G^n(\mathcal S^N) = \cG_N^n(\cS^N)$ denote the special case $K=N$, which is chosen precisely such that the scaling in $N$ of the $\cG^n$-norm is independent of $n$.
Namely, in~\cite{BGJ18a}, it was shown that for every $n$, the $\mathcal G^n$-norm of ${H_0}/{N}$ is order $1$ in $N$. Recall $\xi(t)=\sum a_p t^p$ from~\eqref{eq:H-cov} and that $\xi(1+\epsilon)<\infty$ implies that $H_0\in C^\infty(\cS^N)$.  

\begin{theorem}[{\cite[Theorem 3.3]{BGJ18a}}]\label{thm:reg}
For every $n$, there exist $K(\xi,n),c (\xi,n)>0$ such that
 $H_0/N$ is in $\cG^n$ uniformly in $N$ with high probability: for every $r>0$,
\begin{equation}\label{eq:g-norm-bound}
\prob\big(\norm{H_0}_{\cG^n} \geq K N + r \big) \lesssim \exp(-c r^2/N)\,.
\end{equation}
\end{theorem}
\begin{remark}
The result was stated there for $H_0$ with only one nonzero $a_p$, i.e., the $p$-spin model, however, 
as observed in \cite[Remark 3.4]{BGJ18a} it easily extends to this setting by Borell's inequality
and the fact that in that case, the corresponding $K(p,n)$ is of at most polynomial growth in $p$.
\end{remark}

As further motivation for the definition of the norm $\cG$, specifically with $K=N$, we note here the following 
easy observations which are useful in bounding the regularity of observables with respect to 
Langevin-type operators: we call these the \emph{ladder relations} for $\cG$.
\begin{lemma}[Ladder relations]
\label{lem:Ladder} For every $n\geq 1$, there exists $c(n)$ such that for every $N$,
\begin{equation}\label{eq:ladder}
\begin{aligned}
\norm{\Delta}_{\cG^{n}\to\cG^{n-2}} & \leq1\,, \\
\norm{\left\langle \nabla g,\nabla\cdot\right\rangle }_{\cG^{n}\to\cG^{n-1}} & \leq\frac{c}{N}\norm{g}_{\cG^{n}}\,.
\end{aligned}
\end{equation}
In particular, if $L=\Delta+\left\langle \nabla g_{1},\nabla\cdot\right\rangle $
and $A=\left\langle \nabla g_{2},\nabla\cdot\right\rangle $ for some $g_{i}$
that satisfy $\norm{g_{1}}_{\cG^{2\ell}}\leq c_{1}N$ and $\norm{g_{2}}_{\cG^{1}}\leq c_{2}N,$
then there are $\tilde{c}_1(n,c_{1})>0$ and $\tilde c_2 (n,c_1,c_{2})>0$ such that
\begin{align}\label{eq:ladder-second}
\norm{L^\ell}_{\cG^{n+2}\to\cG^n}\leq \tilde{c}_1\,,\qquad \text{and}\qquad
\norm{AL^{\ell}}_{\cG^{2\ell+1}\to L^{\infty}}\leq \tilde c_2\,.
\end{align}
\end{lemma}
\begin{proof}
The first result in~\eqref{eq:ladder} follows from the fact that traces commute with covariant derivatives. 
Indeed for $f$ smooth, observe that 
\[
\abs{\nabla^k\Delta f} = \abs{\tr(\nabla^{k+2}f)}\leq N\abs{\nabla^{k+2}f}
\]
so that $\norm{\Delta f}_{\cG^n}\leq \norm{f}_{\cG^{n+2}}$.

To see the second inequality in~\eqref{eq:ladder}, observe that if $f\in\cG^{n}$, then 
\[
\norm{|\nabla f|}_{\mathcal G^{n-1}}\leq\frac{1}{\sqrt{N}}\norm{f}_{\cG^{n}},
\]
so that the inequality follows from Cauchy--Schwarz and repeated differentiation. 
Eq.~\eqref{eq:ladder-second} then follows directly from~\eqref{eq:ladder}. 
\end{proof}

As a result of the above and explicit calculation, one also sees the following. 

\begin{corollary}\label{cor:other-glass-norms}
For every $n$, there exist $K(\xi,n)$ and $c(\xi,n)>0$ such that for every $r>0$, 
\begin{align*}
\mathbb P \big(\|\Delta H_0 \|_{\mathcal G^n} \geq K N +r \big) & \leq \exp (-cr^2 /N) \qquad \mbox{and}\\ 
\mathbb P \big(\| |\nabla H_0|^2 \|_{\mathcal G^n} \geq K N +r \big) & \leq \exp (-cr^2 /N)\,.
\end{align*}
\end{corollary}

Finally, an explicit computation also shows that $\phi=\phi_k$ always lives in the space $\mathcal G^m$ for any $m$.
In particular, for every $k$ and every $n$, there exists $K(n,k)>0$ such that 
\begin{equation}\label{eq:phi-bound}
\norm{m_N}_{\mathcal G^n}\leq 1 \qquad \text{ and } \qquad \norm{\phi}_{\mathcal G^n}\leq K\,.
\end{equation}

\subsection{The Langevin operator and existence of the martingale solution} 
Let us now recall some elementary results from stochastic analysis.
For a function $g(x)$,  we always let
$g(t)=g(X_{t})$ denote its evolution under the Langevin dynamics \eqref{eq:X-def}. 
We also let $M_{t}^{g}$ denote the martingale  part of this evolution,
\[
M_{t}^{g}=g(t)-g(0)-\int_0^t L g(s) ds.
\]
Observe that $M_t^g$ is well-defined as the Martingale problem for $L$ given by \eqref{eq:L-def}
is well-posed. 

Let us now recall the following elementary estimate.
Suppose that $f$ is smooth and $\norm{f}_{\cG^{1}}\leq K$; then
by Doob's maximal inequality,
\begin{equation}\label{eq:Martingale-bound}
\sup_x Q_x\Big(\sup_{t\leq T}\big| M_{t}^{f} \big|\geq\epsilon\Big)\lesssim \exp\Big({-\frac{N\epsilon^{2}}{K^{2}T}}\Big)\,.
\end{equation}
As we will frequently use the following estimate, we note that in the case that $m(x)=x_1/\sqrt{N}$,
one has by \eqref{eq:phi-bound}, and Doob's maximal inequality, that there is a universal $K$ 
such that for every $\gamma,T>0$, and every $N$,
\begin{equation}\label{eq:Mag-norm-bound}
\sup_x Q_x \Big(\sup_{t\leq T}\big|{M_{t}^{m}}\big|\geq\frac{\gamma}{\sqrt{N}}\Big)\leq K \exp\Big({-\frac{\gamma^{2}}{K T}}\Big)\,.
\end{equation}
We also define here the following notation which is used throughout the paper. 
Let $\cF_1$ be given by 
\begin{equation}\label{eq:Lm}
\cF_1 = L m = \beta\lambda k m^{k-1}(1-m^2) + L_0 m.
\end{equation}

\section{On Weak and Strong recovery}\label{sec:performance-diverging}
As mentioned in the introduction, there are two main notions of recovery that we study in this paper:
weak and strong recovery. In this section, we discuss the relationship between these.
We prove that  weak recovery implies strong recovery in the diverging $\lambda$ regime. We then show that
depending on the rate of divergence in $\lambda$, there is a certain related radius of correlations, $r_N$, which is $o(1)$,
such that one can weakly recover in order 1 time from \emph{every} initial point with correlation greater than $r_N$. This reduces the difficulty of proving our recovery theorems to showing that the dynamics ``escapes the equator". We end the section observing the stability of weak recovery.

\subsection{Weak recovery implies strong recovery}\label{sec:weak-implies-strong}

We show that as long as $\lambda$ is diverging, 
weak recovery with Langevin dynamics implies strong recovery. 
In the following, we let $\Lambda>0$ be such that $\norm{H_0}_{\cG^1}\leq \Lambda N$ eventually $\mathbb P$-almost surely.
Recall that such a $\Lambda$ exists by \prettyref{thm:reg}.
In this section, for any $\theta\in[-1,1]$, we let $\tau_\theta$ denote the first hitting time for the set
\[
\{x\in \cS^N:m_N(x) =\theta\}\,.
\]

\begin{lemma}
\label{lem:weak-to-strong} 
Fix $k$ and $\beta>0$. For every $\epsilon>0$ and every sequence $\lambda_N \to \infty$, there exists  $T_0>0$ such that for all $T\geq T_0$,  
\[
\lim_{N\to\infty}\, \inf_{x:m(x) \geq \epsilon} \color{black}Q_{x}\bigg(\inf_{t\in[T_0 ,T]}m(t)\geq (1-\eps)\bigg)=1\, \qquad \mbox{$\mathbb P$-a.s.}
\]
\end{lemma}
\begin{proof}
Fix $\epsilon>0$ and suppose that $\lambda_N$ is any diverging sequence. Let $\tau_{\epsilon/2}$ be the first hitting time of $\{m(x) = \epsilon/2\}$. We wish to show that there exists $T_0$ such that uniformly over all $m_N(X_0)\in (\epsilon, 1-\epsilon/2)$, we have that $\tau_{1-\epsilon/2} \leq T_0$ and for every $T$, $ \tau_{\epsilon/2}\geq T$ with $Q_x$--probability $1-o(1)$. 
 For such initial data, eventually $\mathbb P$-almost surely, $\cF_1$ from
 \eqref{eq:Lm} satisfies
\[
\cF_1(t)\geq-1+2^{-k+1} \beta\lambda\epsilon^{k-1}(\epsilon -\epsilon^2/4)-\beta \Lambda\,,
\]
for all $t\leq \tau_{\epsilon/2}\wedge  \tau_{1-\epsilon/2}$.  
Thus, for any $\epsilon,k,\beta>0$ there is a $\lambda_0>0$ such that for every $\lambda>\lambda_0$, 
$\cF_1 (t) \geq c_\lambda>0$ for all $t\leq \tau_{\epsilon/2}\wedge \tau_{1-\epsilon/2}$.
Applying \eqref{eq:Mag-norm-bound} with $\gamma =\sqrt{N}\epsilon/2$, we see that for every $T$, 
\begin{align*}
\inf_{x:m(x) \in (\epsilon, 1-\epsilon/2)} Q_x \left(m({t}) \geq \frac{\epsilon}{2} + c_\lambda t \mbox{ for all $t\leq \tau_{1-\epsilon/2}\wedge T$}\right)\geq 1-Ke^{-N\epsilon^2/KT},
\end{align*}
for some universal $K>0$.
As a consequence, there exists a $T_0(\epsilon,\beta,\lambda)>0$ such that $\tau_{1-\epsilon/2}\leq T_0$. By similar reasoning, for every $T$, 
\begin{align*}
\inf_{x:m(x)\geq 1-\epsilon/2} Q_x \bigg(m(t) \geq 1-\epsilon \mbox{ for all $t\leq T$}\bigg) \geq 1-Ke^{-N\epsilon^2/KT}\,.
\end{align*}
The strong Markov property and a union bound over the above two estimates then implies the desired. 
\end{proof}

\subsection{Weak recovery from microscopic scales}
By a similar argument to the preceeding, one can show that in this regime, weak recovery
occurs as soon as $X_t$ has crossed a certain microscopic correlation. More precisely,
we obtain the following.

\begin{theorem}
\label{thm:spike-dn} 
Fix $k,\beta>0$ and $\lambda_N\to\infty$. There exist $\gamma_0,C>0$ such that if $r_N\to0$ satisfies \[
\lambda_Nr_N^{k-1}=\gamma \qquad\text{ and }\qquad r_N \geq \frac{C}{N^{1-\delta}}\,,
\]
for some $\delta>0$ and some $\gamma\geq \gamma_0$, then for every $\eta>0$,  there is a $T_0>0$ such that for all $T$,
\[
\lim_{N\to\infty} \inf_{x:m_N(x) \geq r_N} Q_x\left(\inf_{t\in[T_0 ,T]}m_N(t)\geq (1-\eta)\right)=1\qquad \mbox{$\mathbb P$-a.s.}
\]
\end{theorem}

The proof of Theorem~\ref{thm:spike-dn} follows from Lemma~\ref{lem:weak-to-strong} and the following lemma showing that the initial boost when $m_N(X_0) \geq r_N$ is sufficiently large to ensure weak recovery. 

\begin{lemma}
\label{lem:(Short-time)} Fix $k,\beta$. Suppose that $r_{N}\to 0$ with $\lambda_{N}r_N^{k-1}=\gamma>0$, 
 and that
 $\gamma,\theta,$ and $r_N$ satisfy
\begin{align*}
-\theta+\frac{k}{2^{k-1}}\gamma\beta(1-\theta^{2})-\beta \Lambda & =:c_\gamma>\theta\,, \qquad \mbox{and} \qquad c_\gamma r_N \geq \frac{1}{N^{1-\delta}}\,
\end{align*}
for some $\delta\in (0,\frac 12)$. Then $\prob$-almost surely,
\[
 \lim_{N\to\infty} \, \inf_{x: m(x)\in (r_N,\theta)}Q_{x}(\tau_{\theta}\leq1)=1\,.
\]
\end{lemma}

\begin{proof}
Let $\tau_{r_N/2}$ be the first hitting time of the set 
$\{x:m(x)= r_N/2\}$.
We wish to show that for all $X_0$ with $m(X_0)\in (r_N,\theta)$, we have  $\tau_{r_N/2} \geq \tau_\theta$ and moreover, $\tau_\theta \leq 1$. We first claim that $\tau_{r_N/2} \geq\tau_{\theta}\wedge1$.
To see this, observe that eventually $\mathbb P$-almost surely, $\cF_1$ from \prettyref{eq:Lm} satisfies
\[
\cF_1(t) \geq-\theta+\frac{k}{2^{k-1}}\gamma\cdot \beta(1-\theta^{2})-\beta \Lambda=  c_\gamma >0\,,
\]
for $t\leq\tau_{r_N/2} \wedge\tau_{\theta}$. By \eqref{eq:Mag-norm-bound},
applied with $\gamma =\sqrt{N} r_N/2$, it follows that for some universal $K>0$, 
\begin{equation}\label{eq:m-bound-t-small}
\inf_{x:m(x)\in (r_N ,\theta)} Q_x\bigg(m({t})\geq \frac{r_N}2+c_\gamma t \mbox{ for all $t\leq\tau_{r_N/2}\wedge\tau_{\theta}\wedge1$}\bigg)\geq 1- Ke^{-N^{2\delta}/4K}\,.
\end{equation}
Since $r_N$ is positive and $c_\gamma>\theta$, we deduce that 
\[
\inf_{x:m(x) \in (r_N,\theta)} Q_x \bigg(\tau_{\theta}\leq \tau_{r_N/2}\wedge 1\bigg)\geq 1-Ke^{-N^{2\delta}/4K}\,,
\]
which implies the desired.
\end{proof}

\begin{proof}[\textbf{\emph{Proof of \prettyref{thm:spike-dn}.}}]
Let $\theta>0$ be sufficiently small. We will first show that from any initial data satisfying $m_N(X_0)\in (r_N,1-\theta)$, we have $\tau_{1-\theta}\leq T_0$ for some $T_0$
with high $Q_x$-probability.
Suppose first that $r_N\leq m_{N}(0)<\theta$.
In this case, \prettyref{lem:(Short-time)} implies that $\mathbb P$-a.s., 
\[
\lim_{N\to\infty} \inf_{x:m(x) \in (r_N ,\theta)} Q_x(\tau_\theta\leq 1) = 1\,.
\]
By the strong Markov property for $X_t$, it 
remains to consider the case  
that $m_{N}(0)\geq \theta$ for some $\theta>0$ sufficiently small. 
By Lemma~\ref{lem:weak-to-strong}, we see that for every $\theta>0$, there exists $T_0$ such that for every $T\geq T_0$, $\mathbb P$-a.s.,  
\[
\lim_{N\to\infty} \inf_{x:m(x) \geq \theta}Q_x\bigg(\inf_{t\in [T_0,T]} m(t) \geq (1-\theta) \bigg)= 1\,,
\]
yielding the result.
\end{proof}

\subsection{Stability of weak recovery} 
\prettyref{lem:weak-to-strong} showed that for $\lambda_N$ diverging and every $\epsilon>0$, if $m(t)$ ever exceeds $1-\epsilon$, then it will remain above $1-2\epsilon$ for all sufficiently large (but order $1$) times. Here we  show an analogous result in the weak recovery regime when $\lambda$ is order $1$ in $N$.  These results can be used to establish ``certificates" for recovery via the Langevin dynamics. 

\begin{lemma}\label{lem:stability}
Fix $k,\beta>0$. For every $\epsilon\in (0,\frac 12)$, there exists a $\lambda_0>0$ such that for all $\lambda\geq \lambda_0$, and every $T$, $\mathbb P$-almost surely, 
\begin{align*}
\lim_{N\to\infty} \sup_{x: m(x) \geq 2 \epsilon} Q_x \bigg(\inf_{t\in [0,T]} m_N(t) \leq \epsilon \bigg)=0\,. 
\end{align*}
\end{lemma}

\begin{proof}
The proof follows by an analogous strategy \prettyref{lem:weak-to-strong}. Fix $k$, $\beta>0$, $\epsilon>0$ and any $T>0$. Let $\tau_{\epsilon}$ and $\tau_{1-\epsilon}$ be the first hitting times of $\{x:m(x) =\epsilon\}$ and $\{x:m(x) =1-\epsilon\}$  respectively and notice that for every $x$ such that $m_N(x)\in (\epsilon,1-\epsilon)$, 
\begin{align*}
\mathcal F_1(x) \geq -1 + k \beta \lambda \epsilon^{k-1} (2\epsilon - \epsilon^2) -\beta \Lambda\,.
\end{align*}
Clearly there exists a $\lambda_0$ sufficiently large such that for all $\lambda>\lambda_0$, the above is positive. Then, by~\eqref{eq:Martingale-bound}, there exists $K>0$ such that for every $s$, 
\begin{align*}
\sup_{x} Q_x \bigg(\sup_{t\leq T} |M_t^m|\geq s\bigg) \leq Ke^{-Ns^2/KT}\,.
\end{align*}
Setting $s= \epsilon/2$, and using the fact that $Lm(x)=\mathcal F_1 (x)$, we see that for every $T$, we obtain the desired inequality. 
\end{proof}

\section{Recovery for Langevin dynamics under Conditions 1 and 2}
\label{sec:randomized-recovery}
We turn now to the proof of the main results of this paper, namely those from \prettyref{sec:general-recovery}.
We first  recall an elementary comparison inequality which will be at the heart
of our comparison between $m_N(X_t)$ and the related gradient flow. We then provide a 
stochastic Taylor-type bound that allows us to propagate the regularity of the initial data on 
order 1 time scales. We then end this section with the proofs of the main theorems. 

Observe the following elementary comparison inequality.

\begin{lemma}
\label{lem:power-law-comparison} Let $\gamma>0$ with $\gamma\neq1$,
$c>0$, and $f\in C_{loc}([0,T))$ with $f(0)>0$. Suppose there exists $T$ such that $f$
satisfies the integral inequality, 
\[
f(t)\geq a+\int_{0}^{t}cf^{\gamma}(s)ds\,,
\]
for every $t\leq T$ and some $a>0$. Then for $t\geq0$ satisfying $(\gamma-1)ca^{\gamma-1}t<1,$
we have
\[
f(t)\geq a\left(1-(\gamma-1)ca^{\gamma-1}t\right)^{-\frac{1}{\gamma-1}}\,.
\]
If the integral inequality holds in reverse, then the corresponding upper
bound holds. If $\gamma>1$, then $T\leq t_{**}$, where $t_{**}=[(\gamma-1)ca^{\gamma-1}]^{-1}$
is called the blow-up time.
\end{lemma}
\begin{remark}
Observe that the case $\gamma=1$ is excluded. This case is the well-known Gronwall's inequality.
In our setting, this corresponds to the critical regime where $k=2$. 
\end{remark}

\begin{proof}
We prove the lower bound, as the upper bound is identical. Furthermore,
it suffices to take $f\in C_{loc}^{1}([0,T))$, as otherwise it suffices
to bound $g(t)=f(0)+\int_{0}^{t}cf^{\gamma}ds$. Since $\gamma>0$,
$F(x)=x^{\gamma}$ is locally Lipschitz on $(\epsilon,\infty)$ for
any $\epsilon>0$. Thus the equation 
\[
\begin{cases}
\dot{h}=ch^{\gamma}\\
h(0)=a
\end{cases}
\]
has as unique solution
\[
h(t)=f(0)\left(1-(\gamma-1)cf^{\gamma-1}(0)t\right)^{-\frac{1}{\gamma-1}}
\]
until a blow up time $t_*$ given by the solution to $(\gamma-1)cf(0)^{\gamma-1}t_{*}=1$. We use the convention that if $t_{*}<0$ then we take $t_{*}=\infty$,
and the solution is global-in-time. Thus, since $f(0)\geq h(0),$
$f$ satisfies $f(t)\ge h(t)$ as desired. 
\end{proof}

\subsection{A growth estimate for Langevin dynamics under gradient-type perturbations.}

In this section, we seek to estimate the growth of well-behaved observables
under the evolution of some Markov process whose infinitesimal generator
is a perturbation of Langevin dynamics for
a sufficiently regular Hamiltonian. 
For every $\delta>0$, we let $B_{\delta}=[-N^{-\frac 12 +\delta},N^{-\frac 12 + \delta}].$ 
\begin{theorem}
\label{thm:taylor-bound}
Let $E\subset\cS^{N}$, $L$ be the infinitesimal
generator of an Ito process $X_{t}$, $f$ be smooth, and $x_0\in E$, with exit time $\tau_{E^c}$.
Suppose that these satisfy the following for some $n\geq 1$.
\begin{enumerate}
\item $L$ is a differential operator of the form $L=L_{0}+a(x)A$ where:
\begin{enumerate}
\item $A$ is of a gradient type: $A=\left\langle \nabla\psi,\nabla\cdot\right\rangle $
for some $\psi\in C^{\infty}$with $\norm{\psi}_{\cG^{2n}}\leq c_{0}N$,
\item $a\in C(\cS^{N})$, with $\norm{a}_{\infty}\leq c_{1}$,
\item $L_{0}=\Delta+\left\langle \nabla U,\nabla\cdot\right\rangle $ for
some $U\in C^{\infty}$ with $\norm{U}_{\cG^{2n}}\leq c_{2}N$.
\end{enumerate}
\item f is smooth with $\norm{f}_{\cG^{2n}}\leq c_{3}$.
\item There exist $\delta>0$ such that the initial $x_0$ satisfies
$L_{0}^{\ell}f(x_0)\in B_{\delta}$ for every $0\leq \ell\leq n-1$.
\item There is an $\epsilon\in(0,1)$ and a $t_{0}$, possibly depending
on $\epsilon$, such that for any $t\leq\tau_{E^c}\wedge t_{0}$, 
\[
\int_{0}^{t}\abs{a(X_{s})}ds\leq\epsilon\abs{a(X_{t})}.
\]
\end{enumerate}
Then there exists $K>0$ depending
only on $c_{i}$ and $\delta$ such that for every $T_0$, 
\begin{equation}
\abs{f(X_{t})}\leq K\left(\frac{N^\delta}{\sqrt{N}}\sum_{\ell=0}^{n-1}(1+t^{\ell})+t^{n}+\frac{1}{1-\epsilon}\int_{0}^{t}\abs{a(X_{s})}ds\right)\label{eq:f-bound}
\end{equation}
for all $t\leq\tau_{E^c}\wedge t_{0}\wedge T_0$ with $Q_{x_0}$-probability $1-O(\exp(-N^{2\delta}/(KT_0)))$.

If instead they satisfy the above with item (3) replaced by
\begin{enumerate}
\item[(3')] There exists $t_1,\delta>0$ such that the initial $x_0$ satisfies
${e^{tL_0} f(x_0)}\in B_\delta$ for every $t<t_1$,
\end{enumerate} 
then the same~\eqref{eq:f-bound} would hold for every
for all $t\leq \tau_{E^c}\wedge t_0\wedge t_1\wedge T_0$.
\end{theorem}

\begin{proof}
For any function $g$, let $g(t)=g(X_t)$. We begin by claiming that $f$ has expansion 
\begin{align}
f(t) & =f(0)+M_{t}^{f}+\sum_{\ell=1}^{n-1}\int_0^t\cdots\int_0^{t_{\ell-1}} L_{0}^{\ell}f(0)+M_{t_{\ell}}^{L_{0}^{\ell}f}dt_{\ell}\cdots dt_{1}\nonumber \\
 & \:\:\qquad+\int_0^t\ldots\int_0^{t_{n-1}} L_{0}^{n}f(t_n)dt_{n}\ldots dt_{1}+\sum_{\ell=1}^{n}\int_0^t\cdots\int_0^{t_{\ell-1}} a(t_\ell\color{black})AL_{0}^{\ell-1}f(t_\ell\color{black})dt_{\ell}\ldots dt_{1}\,.\label{eq:expansion-f}
\end{align}
The proof is by induction. The base case, $n=1$, is simply the definition
of $M_t^f$. Assume that the result holds in the $n$-th case. Then for
the $n+1$-st case we may expand the second-to-last term as
\begin{align*}
\int_0^t \cdots\int_0^{t_{n-1}} L_{0}^{n}f(t_n)dt_{n}\ldots dt_{1} & =\int_{0}^{t}\cdots\int_0^{t_{n-1}} L_{0}^{n}f(0)+M_{t_{n}}^{L_{0}^{n}f}dt_{n}\cdots dt_{1}\\
 & \,\,\qquad+\int_{0}^t\cdots\int_{0}^{t_n} L_{0}^{n+1}f(t_{n+1})+a(t_{n+1})\color{black}AL_{0}^{n}f(t_{n+1})dt_{n+1}\ldots dt_{1},
\end{align*}
by the definition of $M^{L_0^nf}$ and the splitting $L=L_0 + a(x)A$. Combining the terms yields the desired expression,
by induction.

To obtain~\eqref{eq:f-bound} under assumptions (1)--(4), bound the absolute values, term-by-term, in~\eqref{eq:expansion-f}. We first observe
that by the second assumption and the ladder relations \eqref{eq:ladder-second},
\[
\norm{L_{0}^{\ell}f}_{\cG^{2n-2\ell}}\leq\norm{f}_{\cG^{2n}}\leq c_{3}\,.
\]
In particular, $\norm{\nabla L_{0}^{\ell}f}_{\infty}\leq c_{3}/\sqrt{N}$.
Thus by \prettyref{eq:Martingale-bound}, 
with $Q_{x_0}$-probability $1-O(\exp(- N^{2\delta}/(KT_0)))$,
\[
\sup_{t\leq T_0}\abs{M_{t}^{L_{0}^{\ell}f}}\leq\frac{N^\delta }{\sqrt{N}}\,.
\]
Meanwhile, by the third assumption, we can bound $|L_0^{\ell}f(0)|$ for every $0\leq \ell \leq n-1$ by $N^{-\frac{1}{2}+\delta}$. 
Integrating these two inequalities implies that the first line of \eqref{eq:expansion-f} is upper bounded in absolute value by
the first sum in \eqref{eq:f-bound}. 
The second term in~\eqref{eq:f-bound} bounds the integral of $L_0^{n} f (t_n)$ by the ladder relation~\eqref{eq:ladder}.
For the last term in~\eqref{eq:f-bound}, note that by \prettyref{lem:Ladder}, $\abs{AL_{0}^{\ell}f}\leq c$
for some $c>0$; thus the bound follows by applying the fourth assumption
to obtain
\[
\int_{0}^t \ldots\int_0^{t_{\ell-1}}\abs{a(X_{t_\ell})}d_{t_\ell}\cdots dt_1 \leq\epsilon^{n}\abs{a(X_t)}\,. 
\]

To replace assumption (3) by (3'), the same argument applies except we bound the third term in the first line of \eqref{eq:expansion-f} instead as follows.
By Taylor expanding and applying \eqref{eq:ladder} again, 
\[
\abs{e^{t L_0} f(x_0) -\sum_{\ell =0}^{n-1}L_0^\ell f(x_0) \frac{t^\ell}{\ell!}}\leq \norm{L_0^{n} f}_\infty t^n \leq c_3 t^n.
\]
Combining this with assumption (3'), gives the desired bound.
\end{proof}

\subsection{Proof of \prettyref{thm:subcritical-main} for finite $\beta$}
Fix $n\geq1$ and $T_0>0$. For any $\eps>0$, let $\tau_\eps$ denote the hitting time of $\{x: m(x) \geq \eps\}$. For a fixed $\mu_N$, every $\epsilon, \gamma>0$, and every sequence $A_N\subset \cS^N$,
\begin{equation}\label{eq:splitting-subcritical-main}
\int Q_x(\tau_{2\eps} \geq T_0)d\mu_N  \leq \mu_N(A_N^c)+ \mu_N(x_1<\gamma) 
+ \int Q_x(\tau_{2\eps}\geq T_0)\mathbf 1(A_N \cap\{x_1\geq \gamma\}) d\mu_N\,.
\end{equation}
If we take the limit superior in $N$, the first term vanishes almost surely
by \CI~ at level $n$ upon choosing $A_N=E_{n,\delta,N}$ and by the weak form of \CI~ at level $\infty$, upon choosing $A_N=\tilde{E}_{T,\delta,N}$ for any order one $T$. For any fixed $\eta>0$, it follows by \CII~
that there is a $\gamma$ sufficiently small such that the second term is less that $\eta/2$.
Finally, by the following theorem, for any such choice of $\gamma$ sufficiently small, there is an $\epsilon_0>0$ such that for every $\epsilon<\epsilon_0$, the third term 
is less that $\eta/2$ as well. Thus the process reaches $m_N(X_t)\geq 2\epsilon$ in time $T_0$ with probability at least $1-\eta$.
To conclude the proof, note that by Lemma~\ref{lem:stability}, after $\tau_{2\eps}$, $m_N(X_t)$ remains above $\eps$.\qed

\begin{theorem}
\label{thm:k<k_c-main-step} Let $k<2$. For any $\eps>0$, let $\tau_\eps$ denote the hitting time of the set $\{ x: m(x) \geq \epsilon\}$.
For every  $\beta>0$ and $n\geq 1$,  there exist $\lambda_0,\eps_0,c,K>0$ such that for every $\lambda>\lambda_0$, every $\gamma\in(0,1)$, and  every sequence $\mu_N\in\cM_1(\cS^N)$, the following holds:
\begin{itemize}
\item if $k<k_c(n)$ and  $\eps<\eps_0$, then {for every $\delta>0$ sufficiently small}, 
\[
\int Q_x(\tau_{\eps} \geq c\lambda^{-1} \eps^{2-k})\mathbf 1(E_{n,\delta,N}\cap \{x_1\geq \gamma \})d\mu_N(x)\leq K \exp\Big(-\frac{\lambda \gamma^2}{K\eps^{2-k} }\Big)\,,
\]
eventually $\mathbb P$-almost surely.
\item if $k<k_c(\infty)$ and $\eps<\eps_0$, then {for every $T>c\lambda^{-1}\epsilon^{2-k}$ and sufficiently small $\delta>0$}, 
\[
\int Q_x(\tau_{\eps} \geq c\lambda^{-1} \eps^{2-k})\mathbf 1(\tilde{E}_{T,\delta,N}\cap \{x_1\geq \gamma \})d\mu_N(x)\leq K \exp\Big(-\frac{\lambda \gamma^2}{K\eps^{2-k} }\Big)\,,
\]
eventually $\mathbb P$-almost surely.
\end{itemize}
\end{theorem}

\begin{proof}
The proofs of the two statements are identical up to minor modifications. We will focus on the proof of the first and explain the modifications necessary for the second throughout the proof wherever the two proofs differ. Throughout the following, for the second statement, we let $n$ denote the smallest $n$ such that $k < k_c(n)$. In particular, if 
$k=k_c(\ell)$, then $n=\ell+1$.

Let $\delta>0$, later to be chosen sufficiently small, and let $A'=A'(\gamma,\delta)$ denote the event that the initial data, $x\sim \mu_N$, is in $E_{n,\delta,N}\cap\{x_1\geq \gamma\}$ (respectively, in $\tilde{E}_{T,\delta,N}\cap\{x_1\geq \gamma\}$).
Let $\theta=2-k$, so that by assumption, $\theta>\frac 1n$. Finally, without loss of generality take $\eps<1/2$.

Let
$\mathscr T_{L}$ be the first hitting time of the bad set 
\[
\left\{ x:\abs{L_0 m(x)} >\frac 12 \beta k \lambda m^{k-1}(x)\right\} .
\]
On the event $A'$,  by continuity of $X_t$, $\mathscr T_L>0$.
Furthermore, by \prettyref{eq:Mag-norm-bound},
it follows that 
there is a $K_0$ such that for every $\gamma$ and $T$, 
\begin{equation*}
Q_{x}\Big(\sup_{t\leq T}\abs{M_{t}^{m}}> \frac{\gamma}{2\sqrt{N}}\Big)\leq K_0 \exp\Big(-\frac{\gamma^2}{K_0 T}\Big)\,.
\end{equation*}
For $T_0$ positive, to be chosen later, let $$A=A(\gamma,\delta,T_0)= A'(\gamma,\delta)\cap \Big\{\sup_{t\leq T_0} |M_t^m|\leq \frac {\gamma}{2\sqrt N}\Big\}\,.$$ 
For the remainder of the proof we restrict our attention to the event $A$. 
By definition of $\mathcal F_1$~\eqref{eq:Lm},
for every $t\leq\mathscr T_{L}\wedge\tau_\eps$,
\[
\frac{1}{2} \lambda\beta km^{1-\theta}(t)\leq\cF_{1}(t)\leq\frac{3}{2}\lambda\beta k m^{1-\theta}(t).
\]
Then, $m(t)$ satisfies the integral inequality, 
\begin{equation}
\frac{1}{2}m(0)+\int_{0}^{t}\lambda c_{2}m^{1-\theta}(s)ds\leq m(t)\leq\frac{3}{2}m(0)+\int_{0}^{t}\lambda c_{1}m^{1-\theta}(s)ds,\label{eq:integral-ineq-m-k}
\end{equation}
for  $t\leq \tau_\eps\wedge\mathscr T_L\wedge T_0$, uniformly in $\lambda$, for some  $c_{1},c_{2}>0$  which depend on  $\beta$ and $k$. 
Thus by Lemma~\ref{lem:power-law-comparison}, we have that 
\begin{equation}
g_{2}(t):=\left[\left(\frac{1}{2}m(0)\right)^{\theta}+\theta \lambda c_{2}t\right]^{1/\theta}\leq m(t)\leq\left[\left(\frac{3}{2}m(0)\right)^{\theta}+\theta\lambda c_{1}t\right]^{1/\theta}= : g_{1}(t),\label{eq:comparison-m-k}
\end{equation}
for all $t\leq\mathscr T_L\wedge\tau_\eps\wedge T_0$. 

As a consequence of this comparison inequality,  it suffices to show that $\mathscr T_L>T_0$,
where we choose $T_0$ to solve the equation
\begin{equation}\label{eq:eps-to-t0}
\epsilon=\frac{1}{2}\left(\theta c_{2}\lambda T_{0}\right)^{1/\theta}.
\end{equation}

To this end, let us first check that $f(t)=L_0 m(t)$
satisfies the conditions of \prettyref{thm:taylor-bound}. Indeed,
if we let $U=H_{0},\psi(x)=\sqrt{N}x_{1}$, and $a(x)=\beta k \lambda m^{k-1}(x)$,
then item (1) is satisfied eventually $\prob$-a.s.
for every $n\geq1$ by \prettyref{thm:reg} and~\eqref{eq:phi-bound}. 
Item (2) is satisfied by \eqref{eq:ladder} combined with \eqref{eq:phi-bound}.
Item (3) (respectively (3') with the choice $t_1 = T$) follows by assumption on the initial data, namely the event $A'(\gamma,\delta)$. 
It remains to check the fourth condition. Note that by the integral inequality, \prettyref{eq:integral-ineq-m-k},
and the comparison \prettyref{eq:comparison-m-k}, whenever $t\leq\mathscr T_L\wedge\tau_\eps\wedge T_0$,
\[
\int_0^t\abs{a(X_{s})}ds\leq\frac{1}{c_{2}}m(t)\leq\frac{1}{c_{2}}\left[\left(\frac{3}{2}m(0)\right)^{\theta}+\theta \lambda c_{1}t\right]^{1/\theta}.
\]
Observe now that, if we let $T_{1}=(\theta\lambda c_2)^{-1}(c_2/(2c_1))^{-1/\theta}$, then for every $t<T_1$, as long as $\beta k \lambda\geq 1$,
\[
\frac{1}{c_{2}}\left(\theta \lambda c_{1}t+\left(\frac{3}{2}m(0)\right)^{\theta}\right)^{\frac{1}{\theta}}<\frac{\beta k\lambda}{2}\left(\theta\lambda c_{2}t+\left(\frac{1}{2}m(0)\right)^{\theta}\right)^{\frac{1-\theta}{\theta}}\leq \frac 12 a(t)\,,
\]
where the second inequality follows from~\eqref{eq:comparison-m-k}.
Thus, the fourth condition is satisfied for every $t\leq\mathscr T_L\wedge\tau_\eps \wedge T_0 \wedge T_{1}$.

As a consequence, applying \prettyref{thm:taylor-bound}, on the event $A$, for every $\delta>0$ 
\begin{equation}
\abs{f(t)}\leq K\left( \frac{N^\delta }{\sqrt{N}} \sum_{\ell=0}^{n-1}(1+t^{\ell})+t^{n}+2 \int_{0}^{t}\abs{a(X_{s})}ds\right),\label{eq:f-bound-k}
\end{equation}
for $t\leq \tau_\eps\wedge\mathscr T_L\wedge T_0\wedge T_1$ with $Q_{x}$-probability at least $1-2K \exp(-N^{2\delta}/KT_0)$, for some  $K>0$, depending only on $k,n$ and $\delta$.

With this in hand, we aim to show the desired lower bound on $\mathscr T_L$, namely that 
$\mathscr T_L \geq 2\epsilon^\theta/(\theta c_2 \lambda)$.
It suffices to show that for all $t\leq T_0$, each term in~\eqref{eq:f-bound-k} is bounded above by $\beta k\lambda m^{k-1}(t)/(4n+2)$. Begin by observing that for every $0\leq s\leq n-1$, for large enough $N$,
\[
\frac{K N^{\delta}}{\sqrt{N}}(1+t^{s}) \leq \frac{\beta k \lambda }{4n+2} m^{1-\theta}(0)<\frac{\beta k\lambda }{4n+2} g_{2}^{1-\theta}(t)\,,
\]
provided $1+t^{s}=o(N^{\frac{\theta}{2}-\delta})$, which certainly holds for all $t\leq T_{0}\wedge T_{1}$ provided $\delta$
is small enough, since $T_0, T_1$ are order $1$ in $N$. Similarly, there exists an $\epsilon_0>0$ (depending only on $n, \theta, K, c_2$, and not $\lambda$) such that if $\lambda \geq 1$, then for every $\epsilon<\epsilon_0$, for every $t\leq \tau_\epsilon$, 
\[
2K \int_0^t\abs{a(s)}ds\leq2 K \cdot\frac{\beta k}{c_{2}}m(t)\leq\frac{\beta k \lambda}{4n+2}m^{k-1}(t)\,.
\]
Finally, $h(t)=K t^{n}$ satisfies 
\begin{align*}
\frac{d^{\ell}}{dt^{\ell}}h(0) & =0 \leq\frac{d^{\ell}}{dt^{\ell}}g_{2}(0)\quad\mbox{for} \quad \ell\leq n-1\,, \quad \mbox{and}\,,\\
\frac{d^{n}}{dt^{n}}h(t) & = K\cdot n! <  C\lambda^{n+1} \left(\theta \lambda c_{2} t+\left(\frac{1}2 m(0)\right)^{\theta}\right)^{(1-n\theta)/\theta}= \frac {d^n}{dt^n} \Big(\frac{\beta k \lambda}{4n+2} g_2 (t)\Big) \,,
\end{align*}
for some $C$ which depends on  $\beta, k, \theta,n$, and $c_2$, provided 
\[
t\leq\frac{1}{2\theta \lambda c_{2}}\left(\frac{C \lambda^{n+1}}{n! K}\right)^{\theta/(n\theta -1)} = T_{2}\,,
\]
since $n\theta>1$ and $m(0)=o(1)$. Thus if we let $t_{1}= T_{1}\wedge T_{2}$,
we see that on the event that $A$ holds, $t_{1}\wedge \tau_\eps\wedge T_0\leq \mathscr T_L$.  For $\lambda$ sufficiently large, depending only on $\theta,\beta$, we see that $t_1=T_1<T_2$
and that $T_1 = c(\theta,\beta,k)/\lambda$ for some such $c$. 
Consequently, there is an $\eps_0(\theta,\beta,k)$ (independent of $\lambda$)
such that if $\eps\leq \eps_0$ and $T_0$ satisfies \eqref{eq:eps-to-t0}, then $T_0 \leq T_1$
and as a result, $\tau_\epsilon \leq T_0$ as desired. 
\end{proof}

\subsection{Proof of \prettyref{thm:supercritical-main} for finite $\beta$ }
Fix $n\geq1$ and $C,T>0$. For any sequence $r_N$, let $\tau_{C r_N}$ is the hitting time for the set $\{ x: m(x)\geq C r_N\}$. For any sequence $\mu_N\in\cM_1(\cS^N)$, any sequence of events $A_N\subset \cS^N$, and any sequence $r_N$ it follows that 
\begin{align*}
\int Q_x(\tau_{C r_N }\geq T_0)d\mu_N(x)&\leq \mu_N(A_N^c)+\mu_N(x_1<N^{-\delta})\\
&+ \int Q_x(\tau_{C r_N }\geq T_0) \mathbf 1({A_N\cap \{x_1\geq N^{-\delta} \}})d\mu_N(x)\,,
\end{align*}
We take the limit superior in $N$ on both sides  and bound these terms one-by-one. 

If $\mu_N$ satisfies \CI~ at level $n$, then the first term goes to zero $\prob$-a.s. 
by taking $A_N = E_{n,\delta,N}$. Otherwise if $\mu_N$ weakly satisfies \CI~ at level $\infty$, then the same is true upon taking $A_N=\tilde{E}_{T,\delta,N}$. 
If \CII~holds, then the second term goes
to zero $\prob$-a.s. as well. Otherwise, if \CIIprime~holds, then it goes to zero in probability.
The third term goes to zero $\prob$-a.s. by the following theorem upon taking $r_N = N^{-\frac{\alpha}{k-1}}$ and $C>0$.
The result will then follow upon taking $C>0$ sufficiently large so that we may apply \prettyref{thm:spike-dn} and applying the strong Markov property. \qed

\begin{theorem}
\label{thm:alpha>alpha_c-main-step} 
Let $k>2$, $\lambda=N^\alpha$, $r_N=N^{-\frac{\alpha}{k-1}}$, and $C,\beta>0$. Take a sequence $\mu_N\in\cM_1(\cS^N)$ and let $\tau_{C r_N}$ denote
the hitting time of the set $\{x: m(x) \geq C r_N\}$. The following hold.
\begin{itemize}
\item For every $n\geq 1$,  if $\alpha>\alpha_c(n)$, then for $\delta$ sufficiently small, 
\[
\int Q_x(\tau_{C r_N }> 1)\mathbf 1({E_{n,\delta,N}\cap \{x_1\geq N^{-\delta} \}})d\mu_N(x)\leq \exp(-cN^\delta)\,,
\]
eventually  $\prob$-almost surely.
\item If $\alpha>\alpha_0$, then for $\delta$ sufficiently small and $T> 1$,
\[
\int Q_x(\tau_{C r_N }> 1)\mathbf 1({\tilde{E}_{T,\delta,N}\cap \{x_1\geq N^{-\delta} \}})d\mu_N(x)\leq \exp(-cN^\delta)\,,
\]
eventually  $\prob$-almost surely. 
\end{itemize}
\end{theorem}

\begin{proof}
The proofs of the two statements are identical up to minor modifications. We will focus on the proof of the first and explain the modifications necessary for the second in parentheses throughout the proof wherever the two proofs differ. Throughout the following for the second statement, we let $n$ denote the smallest $n$ such that $\alpha > \alpha_c(n)$. In particular, if 
$\alpha=\alpha_c(k )$, then $n=k+1$.
 
Let $A'=A'(\delta)$ denote the event that the initial data $X_0$ is in $E_{n,\delta,N}\cap\{x_1\geq N^{-\delta}\}$ (respectively, in $\tilde{E}_{T,\delta,N}\cap \{x_1\geq N^{-\delta}\}$).
Let $\mathscr T_L$ denote the hitting time of the bad set 
\[
\left\{x: \abs{L_0 m(x)} >\frac{1}{2}\beta k \lambda m^{k-1}\right\}\,.
\]
By our assumptions and continuity of $X_{t}$, on $A'$ both $\mathscr T_L>0$.

By \eqref{eq:Mag-norm-bound},
it follows that for every $\iota>0$ and $T<N^{-\iota}$,
\begin{equation}
\inf_x Q_{x}\Big(\sup_{t\leq T}\abs{M_{t}^{m}}\leq N^{-\iota/4}N^{-1/2}\Big)\geq1-K\exp\big(-N^{\iota/2}/K\big)\,.
\end{equation}
Recalling~\eqref{eq:Lm} and setting $\iota = 5\delta$, it follows that on the intersection of this event and $A'$,  
$m_N(t)$ satisfies the integral inequality 
\begin{equation}\label{eq:m-integral-inequality}
\frac{1}{2}m(0)+\int_0^t\cF_{1}(s)ds \leq m(t)\leq \frac{3}{2} m(0)+\int_0^t\cF_{1}(s)ds\,,
\end{equation}
for all $t< N^{-5\delta}$ eventually $\prob$-a.s. Let us call the intersection of these events $A=A(\delta)$.
In the following, we restrict our attention to this event. 

By definition of $\mathscr T_L$,
there are $c_1,c_2$ positive which depend only on $\beta$ and $k$, such that 
\begin{equation}
\lambda c_{2}m^{k-1}(t)\leq\cF_{1}(t)\leq\lambda c_{1}m^{k-1}(t),\label{eq:F-bound-m-alpha}
\end{equation}
for $t\leq \mathscr T_L$.
By \eqref{eq:m-integral-inequality}, \eqref{eq:F-bound-m-alpha}, and \prettyref{lem:power-law-comparison},
we then obtain upper and lower bounds on $m(t)$ of the form
\begin{equation}
g_{2}(t)\leq m(t)\leq g_{1}(t),\label{eq:m-comparison-alpha}
\end{equation}
for any $t\leq\mathscr T_L\wedge N^{-5\delta}$,
where 
\[
g_{i}(t)=a_{N,i}(1-\frac{t}{t_{**}^i})^{-\frac{1}{k-2}}\,.
\]
Here $a_{N,1}=N^{-\frac 12 +\delta}$, $a_{N,2}= \frac 12 N^{-\frac 12 -\delta}$, and $g_i$ have blow-up times
$$t_{**}^i=\big[(k-2)\lambda c_i a_{N,i}^{k-2}\big]^{-1}\qquad \mbox{for $i=1,2$}$$

Since $k>2$, these blow-up times satisfy $t_{**}^{1}\leq t_{**}^{2}<\infty$. 
In particular, it must
be that $$ \mathscr T_L\wedge N^{-5\delta}\leq t_{**}^{2}$$ since $m\leq 1$.
Furthermore, if we let $\theta =  (k-1)/2-\alpha$, then since $t_{**}^{2}=O(N^{-\frac{1}{2}+\theta+\delta(k-2)})$ and $\theta<\frac 12$, 
we have that $t_{**}^2<N^{-5\delta}$ provided $\delta<\delta_0(k,\alpha)$
for some $\delta_0(k,\alpha)$.

With the above in hand, we aim to show that $\tau_{Cr_{N}}\leq \mathscr T_L$
on $A$. Indeed, if this were the case, then we
would have, on $N$, the desired 
\[
\tau_{Cr_{N}}\leq t_{**}^{2}<1\,.
\] 
Suppose by way of contradiction, that $\tau_{Cr_{N}}>\mathscr T_L$.
It would suffice to show that 
\[
t_{*}=(1-(2C)^{-(k-2)}N^{-(k-2)(\frac{\theta}{k-1}+\delta)})t_{**}^{2}<\mathscr T_L,
\]
as we would then arrive at a contradiction
since, by design, $g_2(t_{*}) = C r_{N}$.

To this end, let us observe the following. First, there is a $\delta_0(k,\alpha)$ such
that for every $\eps>0$, $\delta<\delta_0$, and $N$ sufficiently large, on the event $A=A(\delta)$, 
\begin{equation}
\abs{m(t)}\leq\eps\beta k \lambda m^{k-1}(t) \qquad \mbox{for all $ t\leq\mathscr T_L$}\,,\label{eq:m-lambda-eta-bound-alpha}
\end{equation}
since by~\eqref{eq:m-comparison-alpha}, $m(t)$ is lower bounded by an increasing function for which the inequality holds at time zero. 

Second, let us observe that $f(t)=L_0 m$ satisfies the conditions
of \prettyref{thm:taylor-bound}. Indeed, if we let $U=H_{0},\psi(x)=\sqrt{N}x_{1},$
and $a(x)=\beta k\lambda m^{k-1}(x)$, then item (1) is satisfied eventually $\prob$-a.s.
for any $n\geq1$ by \prettyref{thm:reg}.
Item (2) for $f$ follows by \prettyref{thm:reg} again eventually $\prob$-a.s.
for any $n\geq1$. Item (3) (respectively, item (3') with $t_1=T$) follows by definition, on the event $A'$. 
To see the fourth condition, observe that by the lower bound in 
\eqref{eq:m-integral-inequality} and \eqref{eq:m-lambda-eta-bound-alpha}, 
\[
\int_0^t\abs{a(s)}ds \leq \frac{1}{c_2}m(t)\leq \frac{1}{2}\beta k \lambda m^{k-1}(t)
\]
for $t\leq\mathscr T_L $ provided $N$ is sufficiently large, yielding
the fourth condition.  
Thus by Theorem~\ref{thm:taylor-bound} (choosing $T_0 = 1 \gg t_{\star \star}^2$) there exists $K>0$, such that for some $c>0$,
\begin{equation}
\abs{f(t)}\leq K\left(\frac{N^\delta}{\sqrt{N}}\sum_{\ell=0}^{n-1}(1+t^{\ell})+t^{n}+2\int_{0}^{t}\abs{a(X_{s})}ds\right),\label{eq:bound-f-alpha-case}
\end{equation}
for $t\leq \mathscr T_L\wedge 1$ (in the second case, one should also take the minimum with $T$), on $A$, with $Q_x$-probability $1-O(\exp(-cN^\delta))$ eventually
$\prob$-a.s. provided $\delta<\delta_0(k,\alpha)$ . (We note here that for the second case this inequality still holds at $t>0$.) Now, in order to show that $\mathscr T_L>t_{*}$, thereby concluding the proof, it suffices to show that every term in the sum \eqref{eq:bound-f-alpha-case} is less than $\beta k \lambda m^{k-1}(t)/(4n+4)$ {for $t\leq \mathscr T_L$}.

To this end, suppose first that for some $t\leq\mathscr T_L$, some
$s\leq n-1$, and some $K>0$,
\[
\beta k \lambda m(t)^{k-1}\leq(4n+4)K\frac{N^{\delta}}{\sqrt{N}}(1+t^{s})\,.
\]
Because $g_{2}(t)\leq m(t)$ by \prettyref{eq:m-comparison-alpha}, and $g_{2}$
is increasing, it would follow that 
\[
 N^{-\theta-\delta(k-1)}= \lambda g_{2}(0)^{k-1}\leq \lambda g_{2}(t)^{k-1}\leq \frac{2(4n+4)K}{\beta k} N^{-\frac 12 +\delta},
\]
where we use here that $\mathscr T_L \leq N^{-5\delta}<1$. If we choose $\delta$ sufficiently small, this is impossible for
$N$ sufficiently large.
Suppose instead that for some $t\leq\mathscr T_L$ and some $K>0$,
\[
\frac{\beta k}{2}\lambda m(t)^{k-1}\leq 2{K(4n+4)}\int_{0}^{t}\abs{a(s)}ds.
\]
Since $t\leq\mathscr T_L,$ the comparisons \eqref{eq:m-integral-inequality} and  \eqref{eq:F-bound-m-alpha} imply
that
\[
2\int_{0}^{t}\abs{a(s)}ds=2\int_{0}^{t}\beta k\lambda m^{k-1}(s)ds\leq\frac{2}{c_{1}}m(t).
\]
Thus, for some $t\leq \mathscr T_L$, it would have to be that
\[
\beta k \lambda m^{k-2}(t)\leq\frac{2(4n+4)}{c_{1}}K.
\]
Applying the comparison inequality, \prettyref{eq:m-comparison-alpha},
again yields that on $t\leq t_{\star \star}^2$,
\[
{\beta k}\lambda g_{2}(t)^{k-2}\leq\frac{2(4n+4)}{c_{1}}K.
\]
 Since $g_{2}(t)$ is increasing, this would imply that 
\[
{\beta k} N^{\frac{1}{2}-\theta-\delta(k-2)}={\beta k}\lambda g_{2}(0)^{k-2}\leq\frac{2(4n+4)}{c_{1}}K.
\]
As $\theta<\frac{1}{2}$, this cannot happen either for $\delta$ sufficiently small (depending on $\alpha$ and $k$).

Thus the only remaining way for $\mathscr T_L<t_{*}$ would be that for some
$t<t_{*}$, 
\[
G(t) := {\beta k}\lambda g_{2}(t)^{k-1}\leq 2(4n+4)Kt^{n}= : h(t).
\]
Observe that for $N$ sufficiently large and $t\leq t_*$,
\begin{align*}
\frac{d^\ell}{dt^\ell} h(0) &= 0 < \frac{d^\ell}{dt^\ell}G(0) \quad\mbox{for}\quad \ell \leq n-1\,,\quad \mbox{and}\\
\frac{d^n}{dt^n} h(t) &=2(4n+4)K \cdot n! \leq \frac{d^n}{dt^n} G(t)\,,
\end{align*}
where the second inequality follows by an explicit calculation and the fact that $\theta<\frac{1}{2}\frac{n-1}{n}$ as $\alpha>\alpha_c(n)$, as long as $\delta$ is sufficiently small. This again yields a contradiction.

Thus as long as $\delta$ was sufficiently small and $N$ sufficiently large, then on the event $A$, $t_\star \leq \mathscr T_L$ and in particular, $\tau_{Cr_N}\leq \mathscr T_L \leq 1$ with $Q_x$-probability $1-O(\exp(-cN^{\delta}))$ as desired. 
\end{proof}

\subsection{Proofs of Theorem \ref{thm:supercritical-main}-\ref{thm:subcritical-main} for $\beta=\infty$}\label{subsec:gd-proof}
In this section we demonstrate that the above proofs go through with the appropriate modifications in the setting of gradient descent instead of Langevin dynamics. 
Before we begin this proof we note the following useful remark regarding the local well-posedness of $X_t^\infty$ under \CII. 
\begin{remark}\label{rem:k<2}
We note that in the case $1<k<2$ for gradient descent, the equation for $X_t^\infty$ is ill-posed on $\cS^N$
as $\nabla H$ is only $(k-1)$-H\"older for such $k$. That being said on $\{x_1\geq N^{-\delta}\}$,
$\nabla H$ is locally Lipschitz so that it is locally well-defined by the Picard-Lindelof theorem wherever we work. (In fact, by virtue of the proof 
one also obtains global in time existence.) Consequently, by the assumption that \CII~holds, this subtlety can be avoided since we have interpreted $Q_x$ to be zero
on the set $x_1=0$ so that the inequality \eqref{eq:splitting-subcritical-main} still holds. 
\end{remark}

Recall that in this setup we replace the infinitesimal generator $L$ with $L_\infty = - \langle \nabla H, \nabla \cdot \rangle$ and replace $L_0$ with $L_{0,\infty} = - \langle \nabla H_0, \nabla \cdot \rangle$.
 As a consequence, the drift $\mathcal F_1$ for $m_N(x)$ is replaced by 
\begin{align*}
\tilde \cF_1 = L_\infty m = \beta \lambda km^{k-1}(1-m^2) +L_{0,\infty} m\,.
\end{align*}
Since $\tilde \cF_1 \geq \cF_1$ as long as $m_N(x)\geq 0$, we see that the proofs of Lemma~\ref{lem:weak-to-strong} and Theorem~\ref{thm:spike-dn} go through (but now with no martingale terms to control), 
and we have that eventually $\mathbb P$-a.s., for every $\epsilon>0$, if $\lambda,r_N$ are taken as in Theorem~\ref{thm:spike-dn}, there exists $T_0>0$ such that 
\begin{align}\label{eq:weak-to-strong-gd}
\inf_{X^\infty_0:m_N(X^\infty_0)\geq r_N} \, \inf_{t\geq T_0} \, m_N(X^\infty_t) \geq (1-\epsilon)\,,
\end{align}
for all $N$ large enough. We then notice that the stochastic Taylor expansion in Theorem~\ref{thm:taylor-bound} is only simplified if the Ito process $X_t$ is replaced by the gradient descent process $X^\infty_t$, where $L_0$ is replaced by $L_{0,\infty}$; in particular the bound~\eqref{eq:f-bound} holds deterministically for $|f(X^\infty_t)|$. From there it is evident that following the proof of Theorem~\ref{thm:k<k_c-main-step} and \ref{thm:alpha>alpha_c-main-step}, making the appropriate modifications and omitting the $L^\infty$ bounds on $M_t^m$, will yield its desired analog for $X^\infty_t$. 
Combined with~\eqref{eq:weak-to-strong-gd} and the assumptions on $\mu_N$ satisfying \CI~at level $n$, or weakly satisfying \CI~ at level $\infty$, at inverse temperature $\beta=\infty$, and \CII~(resp., \CII') then allows one to conclude the proof of Theorems~\ref{thm:supercritical-main}-\ref{thm:subcritical-main}.

\section{Checking Conditions 1 and 2: Regularity of initial data}
In this section, we provide some natural examples of initial data satisfying \CI~at different levels as well as \CII~or \CII'. Both of these examples are completely independent of the planted signal so they can be viewed as uninformative, while facilitating recovery of the planted signal. The first of these examples, the volume measure on $\cS^N$, is handled in Section~\ref{sec:initial-data-volume}, which we show weakly satisfies \CI~ at level $\infty$ for every $\beta \in (0,\infty)\cup \{\infty\}$. The second of these examples,  the high-temperature Gibbs measure for $H_0$, is handled in Section~\ref{sec:initial-data-Gibbs}, and shown to satisfy \CI~at level $n$ for every $n$. Both of these imply recovery thresholds of exactly $\alpha_c(\infty)=(k-2)/2$ and $k_c(\infty)=2$. 
%

\subsection{Regularity of the initial data under the volume measure}\label{sec:initial-data-volume} In this section, we show Theorem~\ref{thm:condition-check-volume} holds. Let us begin first by recalling the following result
which usually goes by the name of the Poincar\'e lemma (see e.g.,~\cite{Vershynin}).
This lemma and related concentration and anti-concentration estimates will appear frequently in the following. Their proofs are standard and follow from explicit computation of volumes of spherical caps.
We summarize them here.   
\begin{lemma} \label{lem:volume-C2}
The normalized volume measure $dx$ or $d\mbox{vol}$ on $\cS^N$ satisfies the following.
\begin{itemize}
\item (Poincar\'e lemma) if $X$ is drawn from $dx$, then 
\begin{equation}\label{eq:poincare-lemma}
\sqrt N m_N(X) \convdist Z\,,
\end{equation}
where $Z$ is a standard Gaussian random variable.
\item (Concentration) There exists a universal constant $C>0$ such that for every $t>0$, 
\begin{align*}
\operatorname{vol}(|m_N(X)|>t) \leq C \exp\big(-t^2N /C\big)\,.
\end{align*}
\item (Anti-concentration) There exists a universal $C>0$ such that for every $\eps>0$, 
\begin{equation}\label{eq:anti-concentration}
\operatorname{vol}(\abs{m_N(X)}\leq N^{-\frac 12-\eps})\leq C  N^{-\eps}\,.
\end{equation}
\end{itemize}
\end{lemma}

We now wish to check that the normalized volume measure $\operatorname{vol}^+$ on $\cS^N\cap \{x_1 > 0\}$ weakly satisfies \CI~ at level $\infty$. Specifically, we wish to prove the following concentration estimate.

\begin{theorem}
\label{thm:uniform-L_0-2} 
For every $\beta\in(0,\infty)\cup \{\infty\}$ and $T>0$, there exists $C>0$
so that for every $\delta>0$,  
\[
\operatorname{vol}^+\left(\sup_{t\leq T} |e^{tL_{0,\beta}} L_{0,\beta}m_{N}(X)|\geq N^{-\frac{1}{2}+\delta}\right)\leq Ce^{-N^{\delta}/C}\qquad \mbox{eventually $\mathbb{P}$-a.s.}\,.
\]
\end{theorem}

Notice, first of all, that it suffices to prove Theorem~\ref{thm:uniform-L_0-2} under the law $d\vol$ instead of $d\vol^+$ as $\vol^+(A)\leq 2\vol(A)$ for every set $A\subset \mathcal S^N$. Though it should be the case that
any reasonable initialization that is independent of $H_{0}$ and
the direction $e_{1}$ weakly satisfies \CI~ at level $\infty$, our
proof relies in an essential way on the rotational invariance of the uniform
measure on $\mathcal{S}^{N}$.


\begin{proof}[\textbf{\emph{Proof of Theorem~\ref{thm:condition-check-volume}}}]
The fact that the normalized volume measure $\operatorname{vol}$ on $\mathcal S^N \cap \{x_1 \geq 0\}$ satisfies \CII~ follows from Lemma~\ref{lem:volume-C2}, namely~\eqref{eq:poincare-lemma}. For every $\beta\in (0,\infty)\cup \{\infty\}$, the fact that the normalized volume measure weakly satisfies  \CI~at level $\infty$ at inverse temperature $\beta$ follows from Theorem~\ref{thm:uniform-L_0-2}.  
\end{proof}

Towards the proof of Theorem~\ref{thm:uniform-L_0-2}, it will help to treat the gradient descent case of $\beta =\infty$ distinctly from the $\beta\in (0,\infty)$ situation where there is an additional martingale noise driven by the Brownian motion under $Q_x$. Towards this, let $\hat X_t$ denote the Langevin dynamics generated by $L_{0,\beta}$ when $\beta \in (0,\infty)$ and let $\hat Y_t$ be the gradient descent flow generated by $L_{0,\infty}$. 

We will need several preliminary estimates on the gradients $\nabla H_0 (\hat{X}_t)$ and $\nabla H_{0}(\hat{Y}_{t})$.
It will  be important to ensure that for most initial conditions
and realizations of $H$, the norm of the gradient stays bounded away
from zero, so that $\frac{\nabla H_{0}(\hat{X_{t}})}{|\nabla H_{0}(\hat{X_{t}})|}$ and $\frac{\nabla H_{0}(\hat{Y_{t}})}{|\nabla H_{0}(\hat{Y_{t}})|}$
make sense and their regularity in $t$ can be controlled by $\|H_0\|_{\mathcal{G}^{k}}$.
In this section, we let $\mu_N=dx$ denote the uniform measure on $\cS^N$. 
\begin{lemma}
\label{lem:gradient-lower-bound}There exists $c,\Lambda>0$ (independent
of $N$) such that with $\mathbb{P}$-probability $1-O(e^{-cN})$, 
\begin{align*}
|\nabla H_{0}(\hat{Y}_{t})|\geq & |\nabla H_{0}(\hat{Y}_{0})|e^{-\Lambda t}\,\qquad\mbox{for every $\hat Y_0$ and every $t>0$.}
\end{align*}
In particular, if $\hat Y_0 \sim \mu_N$, then with $\mu_N \otimes \mathbb{P}$-probability
$1-O(e^{-cN})$, we have for every $C>0$,
\begin{align*}
|\nabla H_{0}(\hat{Y_{t}})|\geq c\sqrt{N}e^{-\Lambda t}\qquad\mbox{for all \ensuremath{t>0}, and}\qquad & \inf_{t\leq\frac{C}{\Lambda}\log N}|\nabla H_{0}(\hat{Y}_{t})|\geq cN^{-C+\frac{1}{2}}\,.
\end{align*}
\end{lemma}
\begin{proof}
By Theorem~\ref{thm:reg}, there exists $c,\Lambda(\xi)>0$ such that
with $\mathbb{P}$-probability $1-O(e^{-cN})$, we have $\||\nabla^{2}H_{0}(x)|_{op}\|_{\infty}<\Lambda$.
Differentiating in time, we have
\begin{align*}
\frac{d}{dt}|\nabla H_{0}(\hat{Y}_{t})|^{2}= & -2\nabla^{2}H_{0}(\nabla H_{0},\nabla H_{0})(\hat{Y}_{t})\geq-2\Lambda|\nabla H_{0}(\hat{Y}_{t})|^{2}\,,
\end{align*}
so that by Gronwall's inequality we obtain 
\begin{align*}
|\nabla H_{0}(\hat{Y}_{t})|^{2}\geq & |\nabla H_{0}(\hat{Y}_{0})|^{2}e^{-2\Lambda t}\,,
\end{align*}
as desired. The second claim follows from the fact that there exist
$c_{1},c_{2}>0$ such that 
\begin{align}
\mu_N\tensor\mathbb{P}(|\nabla H_{0}(\hat{Y}_{0})|\leq c_{1}\sqrt{N})\leq & e^{-c_{2}N}.\label{eq:initial-gradient-lower-bound}
\end{align}
Specifically, by the rotational invariance of $\mu_N$ and the law of
$H_{0}(x)$, for any fixed $x\in \cS^N$, 
$$\mu_N\tensor\mathbb{P}(|\nabla H_{0}(\hat Y_{0})|<c_{1}\sqrt{N})=\mathbb{P}(|\nabla H_{0}(x)|<c_{1}\sqrt{N})<e^{-c_{2}N}$$
for some $c_1, c_2>0$, as $\nabla H_{0}(x)$ is distributed as an i.i.d.\ Gaussian vector
in dimension $N$ with entries of order one variance $\sigma_{p}>0$.
\end{proof}
\begin{remark}
\label{rem:gradient-lower-bound-langevin}One could prove a corresponding
bound for the Langevin dynamics by controlling the effect of the noise
in the dynamics via standard martingale estimates. For our purposes, towards Theorem~\ref{thm:uniform-L_0-2},
we only need a lower bound on $|\nabla H_0(\hat X_{t})|$ for order one timescales,
and therefore it suffices to use the bounds of the prequel~\cite{BGJ18a}:
one can read off from \eqref{eq:initial-gradient-lower-bound} and
Theorems 1.2 and 1.4 of~\cite{BGJ18a}
that for every $\beta$ and $T>0$, there exists $c>0$ such that
\begin{align}\label{eq:gradient-lower-bound-langevin}
\mu_N\otimes\mathbb{P}\otimes Q_{x}(\inf_{t\in[0,T]}|\nabla H_0(\hat X_{t})|<c\sqrt{N})\leq & 1-O(e^{-cN})\,.
\end{align}
\end{remark}
We also rely on the rotational
invariance of the gradient when $\hat Y_{0}\sim \mu_N$. Note that as
a corollary of the preceding, $\nabla H_0(\hat Y_{t})/|\nabla H_0(\hat Y_{t})|$  
is well-defined modulo a measure zero set of $\mu_N\tensor\mathbb{P}$ and likewise for $\hat X_t$  under $\mu_N\tensor \mathbb P\tensor Q_x$. 
\begin{lemma}
\label{lem:gradient-is-uniform}For every $t\in[0,\infty)$, if $\hat{Y}_{0}\sim\mu$, then 
the laws of $\frac{1}{\sqrt{N}}\hat{Y}_{t}$ and $\frac{\nabla H_0(\hat{Y}_{t})}{|\nabla H_0(\hat{Y}_{t})|}$
under $\mu\otimes\mathbb{P}$ are distributed as $\mbox{Unif}(\mathbb{S}^{N-1}(1))$.
Similarly, if $\hat{X}_{0}\sim\mu$, and $\hat{X}_{t}$ is the Langevin
dynamics for $H_{0}$ at inverse temperature $\beta$, then for every
$t\in[0,\infty)$, the laws of $\frac 1{\sqrt N}\hat{X}_{t}$ and $\frac{\nabla H_{0}(\hat{X}_{t})}{|\nabla H_{0}(\hat{X}_{t})|}$
under $\mu\tensor\mathbb{P}\tensor Q_{x}$ are distributed as $\mbox{Unif}(\mathbb{S}^{N-1}(1))$.
\end{lemma}
\begin{proof}
It suffices for us to show that the law of $\hat{Y}_{t}$ and $\frac{\nabla H_0(\hat{Y}_{t})}{|\nabla H_0(\hat{Y}_{t})|}$
are invariant under rotations. To this end, fix a rotation ${\mathcal O}\in O(N)$.
Let 
\begin{align*}
\tilde{Y}_{0}= \mathcal{O}\hat{Y}_{0} \qquad \mbox{and}\qquad \tilde{H}_{0}(x)= H_{0}(\mathcal{O}^{-1}x)
\end{align*}
and let $\tilde{Y}_{t}$ be the gradient descent on $\tilde{H}_{0}$
started from $\tilde{Y}_{0}$.\textbf{ }Since $\tilde{Y}_{0}$ is
equidistributed with $\hat Y_{0}$ and $(\tilde{H}_{0}(x))_{x}$ is equidistributed
with $(H_{0}(x))_{x}$, $\nabla\tilde{H}_{0}(\tilde{Y}_{t})$
is equal in distribution to $\nabla H_{0}(\hat{Y}_{t})$ for all $t\geq 0$.
Since $\tilde{Y}_{t}={\mathcal O}\hat{Y}_{t}$, we then see that $\hat{Y}_{t}$
is equal in distribution to $\mathcal{O}\hat{Y}_{t}$ as desired.
Plugging in, it also implies that 
\begin{align*}
\nabla\tilde{H}_{0}(\tilde{Y}_{t})= & \mathcal{O}\nabla H_{0}(\mathcal{O}^{-1}\tilde{Y_{t}})=\mathcal{O}\nabla H_{0}(\hat{Y}_{t})\,.
\end{align*}
 From this we deduce that for every $t\in[0,\infty)$ and every rotation
$\mathcal{O}$, 
\begin{align*}
\nabla H_{0}(\hat{Y}_{t})\eqdist & \mathcal{O}\nabla H_{0}(\hat{Y}_{t}).
\end{align*}
Since $|{\nabla H_{0}(\hat{Y}_{t})}|>0$ holds $\mu_N\otimes\mathbb{P}$-almost
surely, it follows that the distribution of $\frac{\nabla H_0(\hat{Y}_{t})}{\abs{\nabla H_0(\hat{Y}_{t})}}$
is uniform on the $(N-1)$-dimensional unit sphere. 

In the case where $\hat X_{t}$ and $\tilde{X}_{t}$ are the Langevin dynamics,
a similar argument applies upon coupling the driving Brownian motions
as $\tilde{B}_{t}=\mathcal{O}B_{t}$ and then the facts
that $\tilde{X}_{t}$ is equidistributed to $\hat X_{t}$ and $\nabla\tilde{H}_0(\tilde{X}_{t})$
is equidistributed to $\nabla H_0(\hat X_{t})$ hold as before, yielding
the desired. 
\end{proof}
With the above two lemmas in hand, we are now in position to prove that the volume measure weakly satisfies \CI~ at level $\infty$.
In particular, it holds on timescales that
are $O(\log N)$ for gradient descent and, due to the sub-optimality of Remark \ref{rem:gradient-lower-bound-langevin},
the Langevin dynamics for timescales that are $O(1)$. We begin with
the proof for the gradient descent, then modify the proof to carry it over to the Langevin dynamics. 

\begin{proof}[\textbf{\emph{Proof of Theorem~\ref{thm:uniform-L_0-2}  at $\beta= \infty$}}]
By definition of the semigroup, 
\begin{align*}
e^{tL_{0}}L_{0}m(\hat{Y}_{0}) & =L_{0}m(\hat{Y}_{t})\,,
\end{align*}
 and therefore it suffices for us to study $L_{0}m(\hat{Y}_{t})$.
Now by definition, 
\begin{align*}
L_{0}m= & -\langle\nabla H_{0},\nabla m\rangle=-\frac{1}{\sqrt{N}}\langle\nabla H_{0},\nabla x_{1}\rangle\,.
\end{align*}
For each $r\in[-1,1]$, define $\mbox{Cap}(e_{1},r):=\{y\in\mathbb{S}^{N-1}(1):\langle y,e_{1}\rangle\geq r\}$.
Taking $r:=N^{-\frac{1}{2}+\delta}$, by Lemma \ref{lem:gradient-is-uniform}, and~\eqref{eq:poincare-lemma} of Lemma~\ref{lem:volume-C2}, at
every fixed time $t>0$, 
\begin{align}
\mu\tensor\mathbb{P}(\langle\tfrac{\nabla H_{0}(\hat{Y}_{t})}{|\nabla H_{0}(\hat{Y}_{t})|},e_{1}\rangle\geq N^{-\frac{1}{2}+\delta})=\vol(\mbox{Cap}(e_{1},N^{-\frac{1}{2}+\delta}))\leq & e^{-cN^{2\delta}}\,.\label{eq:cap-normalized-gradient}
\end{align}

Our goal is to union bound over $T=C\log N$ many discrete times to show that
with high probability, $L_{0}m(X_{t})$ is small for all $t\leq T$.
However, since this is a continuous time process, we define a series
of bad events, on the complement of which, we can push our union bound
through. For constants $\Lambda$ and $\epsilon>0$,
let
\begin{align*}
B_{1}:=  \{\|H_0\|_{\mathcal G^2}\geq \Lambda N \}\,,\quad\
B_{2}:=  \{X_{0}:|\nabla H_{0}(X_{0})|\leq\epsilon\sqrt{N}\}\,.
\end{align*}
On the event $B_{1}^{c}$, for every
$t\geq s\geq 0$,
\begin{align*}
|\nabla H_{0}(\hat{Y}_{t})-\nabla H_{0}(\hat{Y}_{s})|\leq \int_s^t \abs{\nabla^2 H_0(\nabla H_0,\cdot)(u)}du \leq \Lambda^2 \sqrt{N}(t-s)\,.
\end{align*}
As a consequence, under $B_{2}^{c}$, for every $C>0$ and $t\leq C\log N$,
by Lemma~\ref{lem:gradient-lower-bound},
\begin{align*}
\Big|\frac{\nabla H_{0}(\hat{Y}_{t})}{|\nabla H_{0}(\hat{Y}_{t})|}-\frac{\nabla H_{0}(\hat{Y}_{s})}{|\nabla H_{0}(\hat{Y}_{s})|}\Big| & \leq\epsilon^{-2}N^{2C-1}\big||\nabla H_{0}(\hat{Y}_{s})|\nabla H_{0}(\hat{Y}_{t})-|\nabla H_{0}(\hat{Y}_{t})|\nabla H_{0}(\hat{Y}_{s})\big|\\
 & \leq\epsilon^{-2}N^{2C-1}\Big(|\nabla H_{0}(\hat{Y}_{s})||\nabla H_{0}(\hat{Y}_{t})-\nabla H_{0}(\hat{Y}_{s})|\\
 & \qquad\qquad\quad\qquad+|\nabla H_{0}(\hat{Y}_{s})\big||\nabla H_{0}(\hat{Y}_{s})|-|\nabla H_{0}(\hat{Y}_{t})|\big|\Big)\,,
\end{align*}
which in turn, additionally under $B_{1}^{c}$, is at
most 
\begin{align}
\epsilon^{-2}N^{2C-1}\Lambda^2\sqrt{N}(t-s)\||\nabla H_{0}|\|_{\infty}\leq & \epsilon^{-2}\Lambda^3 N^{2C}(t-s)\:.\label{eq:normalized-gradient-change}
\end{align}

Now fix any $T\leq C\log N$ and partition the interval $[0,T]$ into
$\lceil TN^{4C}\rceil$ intervals of size $\frac{1}{N^{4C}}$, $t_{1}:=0,...,t_{\lceil TN^{4C}\rceil}:=T$.
Since $r=N^{-\frac{1}{2}+\delta}$, 
\begin{align*}
d(\mbox{Cap}(e_{1},\tfrac r2),\mbox{Cap}(e_{1},r)^{c})\geq & \frac{1}{2}N^{-\frac{1}{2}+\delta}\,,
\end{align*}
whence in order for $\frac{\nabla H_{0}(\hat{Y}_{t})}{|\nabla H_{0}(\hat{Y}_{t})|}\in\mbox{Cap}(e_{1},\frac{r}{2})$
for some $t\leq T$, by \eqref{eq:normalized-gradient-change} either
it must have been in $\mbox{Cap}(e_{1},r)$ for one of $\{t_{1},...,t_{\lceil TN^{4C}\rceil }\}$
or one of $B_{1}$ or $B_{2}$ must have occured. Therefore,
by a union bound and \eqref{eq:cap-normalized-gradient}, we have
\begin{align*}
\mu\tensor\mathbb{P}\Big(\sup_{t\leq T}\langle\frac{\nabla H_{0}(\hat{Y}_{t})}{|\nabla H_{0}(\hat{Y}_{t})|},e_{1}\rangle\geq\frac{1}{2}N^{-\frac{1}{2}+\delta}\Big)\leq & \lceil TN^{4C}\rceil e^{-cN^{2\delta}}+\mathbb{P}(B_{1})+\mu\tensor\mathbb{P}(B_{2})\,.
\end{align*}
By \prettyref{thm:reg} and Borrell's inequality, an upper bound of
$O(e^{-cN})$ on $\mathbb{P}(B_{1})$ follows
for $\Lambda$ large enough. The matching
bound on $\mu\tensor\mathbb{P}(B_{2})$ follows from \eqref{eq:initial-gradient-lower-bound}
for any $\epsilon<c_{1}$ there. 
\end{proof}
\begin{proof}[\textbf{\emph{Proof of Theorem~\ref{thm:uniform-L_0-2} for $\beta\in(0,\infty)$}}]
Let $E_{Q_x}$ denote expectation with respect to the law of $\hat X_t$, the Langevin dynamics on $H_{0}$. Recall that by
definition of the Markov semigroup $e^{tL_{0}},$
\begin{align*}
e^{tL_{0}}L_{0}m(x)= & E_{Q_x}[L_{0}m(\hat{X}_{t})]\,.
\end{align*}
 Let us break up $L_{0}m$ into its two constituent parts,
\begin{align*}
L_{0}m= & [-\frac{1}{2}\Delta-\langle\nabla H_{0},\nabla\cdot\rangle]m
=  -\frac{1}{2}m-\frac{1}{\sqrt{N}}\langle\nabla H_{0},\nabla x_{1}\rangle.
\end{align*}

(Note that since $\nabla H_0$ is in $T\cS^N$, and that for $x\in\cS^N$ we have $\nabla x_1 = e_1 - (x_1/\sqrt N) x$, 
it follows that $\langle \nabla H_0,\nabla x_1\rangle =\langle\nabla H_0,e_1\rangle$.)
By \prettyref{lem:gradient-is-uniform} and Lemma~\ref{lem:volume-C2},  the quantities $m(\hat{X}_{t})$
and $\langle\frac{\nabla H_{0}(\hat{X}_{t})}{|\nabla H_{0}(\hat{X}_{t})|},e_{1}\rangle$ satisfied the sub-Gaussian tail bounds that for every $r>0$,
\begin{align*}
\mu\tensor\mathbb{P}\tensor Q_x (|m(\hat{X}_{t})|\geq r)= & \mbox{vol}(\mbox{Cap}(e_{1},r))\leq Ce^{-r^2 N/C}\qquad\mbox{and}\\
\mu\tensor\mathbb{P}\tensor Q_x(|\langle\frac{\nabla H_{0}(\hat{X}_{t})}{|\nabla H_{0}(\hat{X}_{t})|},e_{1}\rangle|\geq r)= & \mbox{vol}(\mbox{Cap}(e_{1},r))\leq Ce^{-r^2 N/C}\,.
\end{align*}
Integrating out the Brownian motion, it is straightforwardly seen that the sub-Gaussian tail bounds transfer to $E_{Q_x}[m(\hat{X}_{t})]$
and $E_{Q_x}[\langle\frac{\nabla H_{0}(\hat{X}_{t})}{|\nabla H_{0}(\hat{X}_{t})|},e_{1}\rangle]$ under $\mu \tensor \mathbb P$.


For constants $\Lambda$, $T$, and $\epsilon>0$,
define the bad events,
\begin{align*}
B_{1}:=  \{\| H_{0}\|_{\mathcal G^3}\geq\Lambda\sqrt{N}\}\,,\qquad
B_{2}:=  \{x:\inf_{t\leq T}E_{Q_{x}}|\nabla H_{0}(\hat{X}_{t})|\leq\epsilon\sqrt{N}\}\,.
\end{align*}
We can now express on $B_{1}^{c}\cap B_{2}^{c}$, by
\prettyref{lem:Ladder},
\begin{align*}
|E_{Q_{x}}[m(\hat{X}_{t})-m(\hat{X}_{s})]| &\leq 2|E_{Q_{x}}[\int_{s}^{t}L_{0}m(\hat{X}_{u})du]\big|+2E_{Q_{x}}[|\langle B_{t}-B_{s},e_{1}\rangle|]\\
&\leq  2(t-s)\|L_{0}m\|_{\infty}+\sqrt{t-s}
\leq  2(t-s)(1+\Lambda)+2\sqrt{t-s}.
\end{align*}
At the same time, by the Cauchy--Schwarz inequality, 
\begin{align*}
\Big|E_{Q_{x}}\Big[&\langle\frac{\nabla H_{0}(\hat{X}_{t})}{|\nabla H_{0}(\hat{X}_{t})|},e_{1}\rangle-\langle\frac{\nabla H_{0}(\hat{X}_{s})}{|\nabla H_{0}(\hat{X}_{s})|},e_{1}\rangle\Big]\Big| \\
& =\Big|E_{Q_{x}}\Bigg[\frac{\Big<\nabla H_{0}(\hat{X}_{t})|\nabla H_{0}(\hat{X}_{s})|-|\nabla H_{0}(\hat{X}_{t})|\nabla H_{0}(\hat{X}_{s}),e_{1}\Big\rangle}{|\nabla H_{0}(\hat{X}_{t})||\nabla H_{0}(\hat{X}_{s})|}\Bigg]\Big|\\
 & \leq E_{Q_{x}}\Bigg[\frac{\Big|\langle\nabla H_{0}(\hat{X}_{t}),e_{1}\rangle-\langle\nabla H_{0}(\hat{X}_{s}),e_{1}\rangle\Big|}{|\nabla H_{0}(\hat{X}_{t})|}+\frac{\Big||\nabla H_{0}(\hat{X}_{t})|-|\nabla H_{0}(\hat{X}_{s})|\Big|}{|\nabla H_{0}(\hat{X}_{t})|}\Bigg]\\
 & \leq2\sqrt{\frac{ E_{Q_{x}}|\langle\nabla H_{0}(\hat{X}_{t}),e_{1}\rangle-\langle\nabla H_{0}(\hat{X}_{s}),e_{1}\rangle|^{2}\vee E_{Q_{x}}||\nabla H_{0}(\hat{X}_{t})|-|\nabla H_{0}(\hat{X}_{s})||^{2}}{E_{Q_{x}}|\nabla H_{0}(\hat{X}_{t})|^{2}}}\,.
\end{align*}
On the event $B_1^c\cap B_2^c$, this is at most
\begin{align*}
\epsilon^{-1}N^{-\frac{1}{2}}\Bigg\{E_{Q_{x}}[|&\langle\nabla H_{0}(\hat{X}_{t}),e_{1}\rangle-\langle\nabla H_{0}(\hat{X}_{s}),e_{1}\rangle|]^{2}\vee E_{Q_{x}}[||\nabla H_0(\hat X_{t})|-|\nabla H_0(\hat X_{s})||]^{2}\Bigg\}^{1/2}\\
&\leq  2\epsilon^{-1}N^{-\frac{1}{2}}\Big[(t-s)(\|L_{0}\langle\nabla H_{0},e_{1}\rangle\|_{\infty}\vee \|L_{0}|\nabla H_{0}|\|_{\infty})+\sqrt{(t-s)}\||\nabla^{2}H_{0}|\|_{\infty}\Big]\\
&\leq  K\epsilon^{-1}N^{-\frac{1}{2}}\Big[(t-s)\Lambda+\sqrt{t-s}\Lambda\Big]\,,
\end{align*}
for some constant $K>0$  by~\eqref{eq:gradient-lower-bound-langevin}. Fixing any finite $T$ and partitioning
$[0,T]$ in to $\lceil TN^{4}\rceil$ intervals of width $N^{-4}$,
and arguing as in the proof of the gradient descent estimate, we see
that in order for $e^{tL_{0}}L_{0}m$ to have been at least $\frac{1}{2}N^{-\frac{1}{2}+\delta}$,
either $B_{1}$ or $B_{2}$ occured, or one of $|E_{Q_{x}}[m(\hat{X}_{t})]|\geq\frac{1}{4}N^{-\frac{1}{2}+\delta}$
or $|E_{Q_{x}}[\langle\frac{\nabla H_{0}(\hat{X}_{t})}{|\nabla H_{0}(\hat{X}_{t})|},e_{1}\rangle]|\geq\frac{1}{4}N^{-\frac{1}{2}+\delta}$
for some $t\in\{0,N^{-4},...,T\}$. By a union bound and the above, 
\begin{align*}
\mu\tensor\mathbb{P}\Big(\sup_{t\leq T}|e^{tL_{0}}L_{0}m(x)|\geq N^{-\frac{1}{2}+\delta}\Big)\leq & \lceil TN^{4}\rceil e^{-cN^{2\delta}}+\mathbb{P}(B_{1})+\mu\tensor\mathbb{P}(B_{2})\,.
\end{align*}
By \prettyref{thm:reg} and Borrell's inequality, an upper bound of
$O(e^{-cN})$ on $\mathbb{P}(B_{1})$ follows
for $\Lambda$ large enough. The matching
bound on $\mu\tensor\mathbb{P}(B_{2})$ follows from \eqref{rem:gradient-lower-bound-langevin}
for any $\epsilon<c$ there. 
\end{proof}

\subsection{Regularity of the initial data under the high-temperature Gibbs measure}\label{sec:initial-data-Gibbs}
Recall from the introduction that $d\pi_{0,\beta}(x)\propto e^{-\beta H_{0}(x)}dx$
denotes the Gibbs measure corresponding to the pure noise Hamiltonian, $H_0(x)$, at inverse temperature $\beta>0$, and $\pi_{0,\beta}^+$ is the same Gibbs measure conditioned on $\{x_1 \geq 0\}$.
Observe the following  consequence of \eqref{eq:anti-concentration} and 
the isotropy of  $H_0$.
\begin{lemma}\label{lem:gibbs-cond-2'}
For every $\xi$, every $\beta>0$ and every $\delta>0$, 
\[
\pi_{0,\beta}^+(x_1 \leq N^{-\delta})\to 0\quad \mbox{in prob.}
\]
In particular, $\pi_{0,\beta}^+$
satisfies \CII' for every $\delta>0$. 
\end{lemma}
\begin{proof}
By isotropy of the law of $H_0$ and \eqref{eq:anti-concentration}, we have
\begin{align*}
\E[\pi_{0,\beta}^+(x_1\leq N^{-\delta})] = \mbox{vol}(|x_1|\leq N^{-\delta}) 
\leq & CN^{-\delta/2}\,,
\end{align*}
which after an application of Markov's inequality, implies the desired.
\end{proof}
We now turn to the main estimate in this section. 
\begin{lemma}
For every $\xi$, there exists a $\beta_{0}>0$ such that eventually $\mathbb{P}$-almost surely, for every $\beta<\beta_{0}$, the measure $\pi_{0,\beta}$
satisfies \CI~at level $n$ at inverse temperature $\beta$, for every $n$. Consequently, if $\xi$ is even, the measure $\pi_{0,\beta}^+$ satisfies \CI~at level $n$ for every $n$ as well. 
\end{lemma}
\begin{proof}
We wish to show that there exists $K(\xi,\beta)>0$ such that for every
$\ell\geq0$, eventually $\mathbb{P}$-a.s., 
\begin{align}\label{eq:WTS-gibbs-C1}
\pi_{0,\beta}(|L_{0}^{\ell}m_N(x)|\geq N^{-\frac{1}{2}+\delta})\leq & Ke^{-N^{2\delta}/K}\,,
\end{align}
as long as $\beta$ is sufficiently small. Recall that there is an important implicit dependence on $\beta$ in $L_0=L_{0,\beta}$. 
A union bound over the first $n$ such events would then imply $\pi_{SG}$
satisfies condition 1 at level $n$. 
For the case  $\ell=0$, it suffices to check  
\begin{align*}
\mathbb{E}[\pi_{0,\beta}(|x_{1}|\geq N^{\delta})]\leq  e^{-N^{2\delta}}\,.
\end{align*}
This holds for every $\beta>0$ from the fact that $\mathbb{E}[\pi_{0,\beta}(A)]=\mbox{vol}(A)$
and the concentration of $(X,e_1)$ under the volume measure $dx$. This then implies~\eqref{eq:WTS-gibbs-C1} with $\ell=0$ by Markov's inequality. 

We now proceed to the cases $\ell\geq1$. Recall that for every $\beta$,
the generator $L_0=L_{0,\beta}$ is essentially
self-adjoint in $C^{\infty}(\mathcal{S}^{N})\subset L^2(d\pi_{0,\beta})$. 

Furthermore,  it was shown in \cite{GJ16} that there exists $\beta_{0}>0$ such
that for every $\beta<\beta_{0}$, the measure $\pi_{0,\beta}$ satisfies
a log-Sobolev inequality with constant $c(\xi,\beta)>0$ eventually
$\mathbb{P}$-almost surely. 

By Herbst's argument \cite{Led01}, it then follows that
for every smooth, $1$-Lipschitz function $F$, we have 
\begin{align*}
\pi_{0,\beta}(|F(x)-\pi_{0,\beta}[F]|\geq r)\lesssim e^{-r^{2}/(2\sqrt{c})}\,,
\end{align*}
where $\pi_{0,\beta}[F]$ is the expectation of $F(x)$ under $\pi_{0,\beta}$.
Now fix any $\ell\geq1$ and let $F(x)=\sqrt N L_{0}^{\ell}m_N(x)$: by essential self-adjointness
of $L_{0,\beta}$ with respect to $\pi_{0,\beta}$ and smoothness of $m_N(x)$, we see that
\begin{align*}
\int (L_{0}^{\ell}m) (x)d\pi_{0,\beta}(x)=\langle1,L_{0}^{\ell}m\rangle_{L^{2}(\pi_{0,\beta})} & =\langle L_{0}^{\ell}1,m\rangle_{L^{2}(\pi_{0,\beta})}=0\,.
\end{align*}
At the same time, we claim that $$\|L_{0}^{\ell}m_N\|_{Lip}\leq\||\nabla L_{0}^{\ell}m_N|\|_{\infty}\leq \frac{C_{\xi,\beta,\ell}}{\sqrt N}\,,$$ for some $C_{\xi,\beta,\ell}>0$.
Recall from \eqref{eq:phi-bound} that $\|m_N\|_{\mathcal{G}^{k}}\leq 1$
for all $k$. On the other hand, by~\eqref{eq:g-norm-bound}--\eqref{eq:ladder-second}, for every $n$ and $\ell$, there exists $C_{\xi,\beta,\ell}>0$ such that eventually $\mathbb{P}$-almost surely, $$\|L_{0}^{\ell}\|_{\mathcal{G}^{n}\mapsto\mathcal{G}^{n-2}}\leq C_{\xi,\beta,\ell}\,.$$
Thus, for every $k$, eventually $\mathbb{P}$-almost surely we have $\|L_{0}^{\ell}m_N\|_{\mathcal{G}^{k}}\leq C_{\xi,\beta,\ell}$.
By definition of the $\mathcal{G}^{k}$-norm, it then follows that
 $\||\nabla L_{0}^{\ell}m_N|\|_{\infty}\leq\frac{C_{\xi,\beta,\ell}}{\sqrt N}$. Combining
these, we see that for every $\ell\geq1$,
\begin{align*}
\pi_{0,\beta}(|L_{0}^{\ell}m(x)|\geq {r}N^{-\frac{1}{2}})\lesssim & e^{-r^{2}/2}
\end{align*}
eventually $\mathbb P$-a.s., which implies the desired concentration estimate when plugging in
$r=N^{\delta}$. 
\end{proof}

\section{Free energy wells and obstructions to recovery} 
In Section~\ref{subsec:FEW} we formalize the notion of free energy wells and obtain an elementary exit time lower bound that is exponential in the height of the well, started from the natural Gibbs initialization in the well. Recall from Remark~\ref{rem:eq-pass} that imposing such  an initialization is necessary to have such an exit time lower bound. 

We then show that the measure $d\pi\propto \exp(-\beta H)dx$ has a free energy well around the equator: 

\begin{proposition}\label{prop:x_1-GFEB-2}
Fix $\xi,k$ and $\beta>0$. If $\alpha <\frac{k-2}{2}$ then for $\epsilon \in (0,\frac 12)$ sufficiently small, eventually $\mathbb P$-a.s., the function $f(x)=(x,e_1)$ has a $(\frac 12 N^{\epsilon})$--free energy well of height $N^{2\epsilon}$ at $[-\frac 32 N^{\epsilon},\frac 32 N^{\epsilon}]$.
\end{proposition}

We end the section by combining this with a crude comparison between $\pi$ and $\pi_0$ near the equator, to deduce Theorem~\ref{thm:worst-case-main}.

\subsection{Free energy wells and hitting time lower bounds}\label{subsec:FEW}

The existence of free energy wells is closely related to the behavior of both equilibrium and off-equilibrium Langevin dynamics. In particular, free energy barriers govern the eigenvalues of the Dirichlet problems on domains and in turn the exit times of sets from their interiors, as well as spectral gaps. Importantly, though, these relationships are, in a sense, with respect to worst possible initial starts or with respect to Gibbs--typical initializations.  

For the recovery problem in our setting, what will be relevant are free energy barriers between the equator, or volume--typical latitudes, and the signal.  

In this section we will work in full generality of Langevin dynamics on $\cS^N$ with respect to some Hamiltonian $V_N$.  Let $f:\cS^N \to\mathbb R$ be some smooth, say $C^\infty(\cS^N)$, observable. The Langevin dynamics has invariant measure which we denote (abusing notation) $d\pi (x) \propto \exp(-\beta V_N)$. For the purposes of the recovery problem, one should have in mind the example $V_N = H$ and $f(x)=(x,e_1)$.

For any function $f:\cS^N \to \mathbb R$, we can define the following ``rate function": 
\begin{align*}
I_f(a;\epsilon)= - \log \pi(\{x: f(x)\in B_\epsilon (a)\})\,.
\end{align*}

\begin{definition}\label{def:FEW}
The function $f$ has an $\epsilon$--\emph{free energy well} of height $h$ in $[a,b]$ if the following holds: there exists $c\in (a,b)$ and a pair $\epsilon,\eta$ such that $B_\epsilon (a), B_\epsilon (b), B_\eta(c)$ are disjoint, and 
\begin{align*}
 \min\{I_f (a,\epsilon),I_f (b,\epsilon)\} -  I_f (c,\eta) \geq h \,.
\end{align*}
\end{definition}

\begin{remark}
It will sometimes be useful to consider one-sided energy wells. To do so, we may take either $a=-\infty$ or $b=\infty$ in which case, for every $\epsilon$, we would set 
\begin{align*}
I_f (\pm \infty ,\epsilon) =\infty\,.
\end{align*}
\end{remark}

\begin{theorem}\label{thm:fe-barrier-metastability}
There is a universal constant $C>0$ such that the following holds. 
Suppose that  $f$ is a smooth function with no critical values in an open neighbourhood of $[a,b]$.
If $f$ has an $\eps$-free energy well of height $h$ in $[a,b]$ with 
\[
\frac{\eps^2}{h} \leq \frac{C \norm{\nabla f}_{L^\infty(A_\eps)}^2 }{N+\sqrt N \beta \norm{ \nabla V}_{L^\infty(A_\eps)}}\,,
\]
where $A_\eps = f^{-1}([a,b])$, then the exit time of $A_\epsilon$, denoted $\tau_{A_\epsilon^c}$, satisfies
\[
\int   Q_x \big(\tau_{A_\epsilon^c}\leq T\big) d\pi\big(x \mid A_\epsilon \big)
\lesssim \Big(1+   \epsilon^{-4}  h T  \norm{\nabla f}_{L^\infty(A_\eps)}^4 \Big)\exp\left(-h\right)\,.
\]
\end{theorem}

\begin{proof}
For ease of notation, let  $A_\eps = f^{-1}([a,b])$, $A = f^{-1}([a,b]\setminus (B_\eps(a)\cup B_\eps(b)))$,
and $B_\eps = A_\eps\setminus A$. Since $[a,b]$ consists of only regular points, $f^{-1}([a,b])$ is a submanifold
with smooth boundary by the preimage theorem \cite{guillemin2010differential}. 
Therefore, Langevin dynamics reflected at the boundary of $A_\eps$ is well-defined in the sense
of the martingale problem \cite{StroockVaradhan06}. Call this process $\bar X_t$ and denote its law by $\bar Q_x$. 
Recall that it is reversible, with stationary measure,
\[
\nu = \pi(\cdot \mid  A_\eps)\,.
\]

By definition, $X_t$ and $\bar X_t$ are equal in law conditionally on $t\leq \tau_{\partial A_\eps}$.
As a result, 
\[
\int Q_x(\tau_{\partial A_\eps}< T)d\nu(x) =\int \bar Q_x(\tau_{\partial A_\eps}<T)d\nu(x)\,.
\]

Let us estimate the right-hand side above. Let $s_0$ be a positive number, to be chosen later, 
and let $t_i = (i\cdot s_0) \wedge T$ for each $i=0,\ldots, \lceil T/s_0\rceil$.  Then 
\[
\int_{A} \bar Q_x(\tau_{\partial A_\eps} < T)d\nu(x)
\leq \int_{A_\epsilon} \bar Q_x\big(\exists i: \bar X_{t_i}\in B_\eps\big)d\nu(x)
+ \int_{A}  \bar Q_x \big(\tau_{\partial A_\eps} < T,\, \bar X_{t_i\wedge \tau_{\partial A_\eps}}\in A \,\, \forall i\big)d\nu(x)\,.
\]
Call the first term above $I$, and the second term $II$, and bound them separately.
By a union bound and reversibility with respect to $\nu$,
\[
I \leq \frac{T}{s_0} \nu(B_\eps)\,.
\]
On the other hand, by the Markov property, the equality in law of $X_{t\wedge\partial A_\eps}$ and $\bar X_{t\wedge\partial A_\eps}$, and a union bound,
\begin{align*}
II &\leq \int_A  \bar Q_x (\exists i : t_i <\tau_{\partial A_\eps}< t_{i+1}, \, \bar X_{t_i} \in A ) d\nu(x) \\
&= \int_A Q_x( \exists i: t_i <\tau_{\partial A_\eps}< t_{i+1}, \, X_{t_i}\in A)d\nu(x)\\
&\leq \frac{T}{s_0} \sup_x  Q_x \Big(\sup_{t\leq s_0} d(\bar X_t,x)\geq \eps/\norm{|\nabla f|}_\infty\Big).
\end{align*}

To bound this last term, observe that for any $x_0\in\cS^N$, 
the function $g(x)=\abs{x-x_0}^2$ is smooth and, by an explicit calculation,
\[
\norm{\Delta g }_{L^\infty(\cS^N)} \leq CN\,, \qquad \mbox{and} \qquad \norm{|\nabla g|}_{L^\infty(\cS^N)}\leq C \sqrt{N}\,,
\]
for some universal constant $C>0$. In particular, 
\[
\norm{Lg}_{L^\infty(A_\eps)}\leq  C(N+ \beta\sqrt N  \norm{\nabla V}_{L^\infty(A_\eps)})\,.
\]
Thus, provided $C(N+ \sqrt N \beta \norm{\nabla V}_{L^\infty(A_\eps)})s_0 \leq \eps^2/(2\norm{\nabla f}_{L^\infty(A_\eps)}^2)$, we obtain
\begin{align}
Q_x \Big(\sup_{t\leq s_0} d(X_t,x)\geq\frac{\eps}{\norm{\nabla f}_{L^\infty(A_\eps)}}\Big) & \leq Q_x \Big(\sup_{t\leq s_0} |M_t^g| \geq  \frac{\epsilon^2}{2\|\nabla f\|_{L^\infty(A_\epsilon)}^2}\Big) \\ 
& \leq K\exp\Big(-\frac{\eps^4}{K\norm{\nabla f}^4_{L^\infty(A_\eps)} s_0}\Big),
\end{align}
for some universal $K>0$ by Doob's maximal inequality~\eqref{eq:Martingale-bound}. If we choose 
\[
s_0 = \frac{1}{Kh}\left(\frac{\eps^2}{\norm{\nabla f}^2_{L^\infty(A_\eps)}}\right)^2,
\]
then by assumption, $s_0$ satisfies this inequality from which it follows that the righthand side of the 
preceding display is bounded by $K\exp(-h)$.

By the assumption of $f$ having an $\epsilon$-free energy well of height $h$, we know that $\nu(B_\eps)\leq \exp(-h)$. Combing these estimates yields
\[
\int Q_x(\tau_{\partial A_\eps}< T)d\nu \lesssim \Big(1+\frac{T}{s_0}\Big)\exp(-h)
\]
which yields the desired bound, after plugging in for $s_0$. 
\end{proof}

\subsection{Spin Glass free energies} 
We now turn to the specific case of an equatorial free energy well for $\pi_\beta$.
Define the following \emph{restricted spin-glass free energy}: for a Borel set $A\subset \mathcal S^N$ define
\begin{align*}
 F_{0}(A)=F_{0,N}(A;\beta) = \frac 1N \log \int_A e^{-\beta H_0(x)}dx\,.
\end{align*} 
Similarly, define the restricted free energy, 
\begin{align*}
F(A)= \frac{1}{N} \log \int_{A} e^{-\beta H(x)} dx\,, 
\end{align*}
and define $F(A)$ analogously. Finally, let  
\begin{equation}\label{eq:A-delta-def}
A_{\delta_N} = \{x: m_N(x) \in (-\delta_N, \delta_N )\}
\end{equation}
denote the equatorial band of correlation at most $\delta_N$.
We begin by showing the following estimate for the spin glass free energy. 

\begin{lemma}\label{lem:sg-free-energy-2}
Fix any $\xi$, and $\beta, \epsilon>0$. Letting $\delta_N = N^{-\frac 12 +\epsilon}$, there exists $C>0$ such that eventually $\mathbb P$-almost surely, 
\begin{align*}
F_0 (A_{\delta_N/2})-F_0 (A_{\delta_N}\setminus A_{\delta_N/2})\geq C N^{-1+2\epsilon}\,.
\end{align*} 
\end{lemma}
\begin{proof}
We will in fact show the stronger   
\begin{align}
F_0 (A_{\delta_N/2})-F_0(A_{\delta_N/2}^c) \geq CN^{-1+2\epsilon}\,,
\end{align}
which would imply the desired since $A_{\delta_N}\setminus A_{\delta_N/2}\subset A_{\delta_N}^c$ and thus $F_0 (A_{\delta_N}\setminus A_{\delta_N/2})\leq F_0(A_{\delta_N/2}^c)$. 
Since for any $A$, 
\begin{align*}
F_0 (A)-F_0(A^c) = \frac 1N \log \frac{Z_{0}(A)}{Z_{0}(A^c)} = \frac 1N \log \bigg(\frac{1}{\pi_{0,\beta}(A^c)}-1\bigg)\,,
\end{align*}
it suffices to show that $\mathbb P$--eventually almost surely, 
\begin{align*}
\pi_{0,\beta} (A_{\delta_N/2}^c)\lesssim  \exp (-cN^{2\epsilon})
\end{align*}
for some $c(\xi,\beta)>0$.

In order to upper bound $\pi_{0,\beta}(A_{\delta_N/2}^c)$, notice that for any fixed $A\subset \cS^N$, 
\begin{align*}
\mathbb E [\pi_{0,\beta}(x\in A)]=\mbox{vol}(A)\,,
\end{align*}
by isotropy of the law of $H_0$. 
On the other hand, the concentration estimate in \prettyref{lem:volume-C2} shows that $\mbox{Vol}(A_{\delta_N/2})\leq e^{-N^{2\epsilon}/8}$
so that by Markov's inequality, 
\begin{align*}
\mathbb P\big(\pi_{0,\beta}(x\in A_{\delta_N/2}^c)\geq e^{-\frac{N^{2\epsilon}}{16}}\big)\leq e^{-\frac{N^{2\epsilon}}{16}}\,.
\end{align*}
Combining the above, we obtain the desired inequality. 
\end{proof}

\begin{proof}[\textbf{\emph{Proof of Proposition~\ref{prop:x_1-GFEB-2}}}]
Let $\delta= \delta_N = N^{-\frac 12 +\epsilon}$ for some $\epsilon$ to be chosen sufficiently small. Define $A_\delta$ as in \eqref{eq:A-delta-def}, and let $B_\delta= A_{2\delta}-A_\delta$. First of all, 
\begin{align*}
F(A_{\delta} ) = \frac 1N \log \int_{A_{\delta}} e^{-\beta(H_{0}(x)-\lambda\phi(x))}dx
 \geq  \frac 1N \log\int_{A_{\delta/2}}e^{-\beta H_{0}(x)}dx=F_0(A_{\delta/2})\,,
\end{align*}
while, on the other hand, 
\begin{align*}
F(B_\delta) & \leq F_0(B_\delta)+\beta\lambda \delta_{N}^{k}\,.
\end{align*}

Thus, 
by Lemma \ref{lem:sg-free-energy-2} and the above inequalities, $\mathbb{P}$-eventually
almost surely, 
\begin{align*}
I_f (\frac{3}2 N^\epsilon; \frac 12 N^\epsilon) - I_f (0; N^{\epsilon}) = F(B_{\delta})-F(A_{\delta})\leq & -\frac{C}{N^{1-2\epsilon}}+\beta\lambda \delta_{N}^{k}\,,
\end{align*}
for some $C(\xi,\beta,k)>0$. But by the assumption that $\alpha<(k-2)/2$, if we choose $\epsilon>0$ sufficiently small, namely $0< \epsilon < \frac1k(\frac{k-2}{2}-\alpha)$, we see that the above satisfies $$I_f (\frac{3}2 N^\epsilon; \frac 12 N^\epsilon) - I_f (0; N^{\epsilon})<-\frac{C'}{N^{1-2\epsilon}}$$ 
for some $C'(\xi,\beta,k)>0$. 
When $k$ is even, by symmetry, we similarly have 
\begin{align*}
I_f (-\frac 32 N^{\epsilon}; \frac 12 N^\epsilon) -  I_f (0; N^{\epsilon}) <-\frac{C'}{N^{1-2\epsilon}}\,,
\end{align*}
and when $k$ is odd, we would have  
\begin{align*}
I_f (-\frac 32 N^{\epsilon}; \frac 12 N^\epsilon) -  I_f (0; N^{\epsilon})\leq -\frac{C}{N^{1-2\epsilon}} -\beta \lambda \delta_N^k
\end{align*}
which is negative regardless of the choice of $\alpha$.
By Definition~\ref{def:FEW}, this concludes the proof.
\end{proof}

\subsection{Proof of Theorem~\ref{thm:worst-case-main}}
We begin by combining Proposition~\ref{prop:x_1-GFEB-2} and Theorem~\ref{thm:fe-barrier-metastability} to obtain an exit time lower bound for the Gibbs measure $\pi$ restricted to the equatorial free energy well. Fix $\beta>0$ and take $V_N = H_N$, $\epsilon_N = \frac 12 N^\epsilon$, $[a,b] = [-\frac 32 N^{\epsilon}, \frac 32 N^{\epsilon}]$ and $h_N = N^{2\epsilon}$ for any $\epsilon$ sufficiently small. By~\eqref{eq:phi-bound}, $\||\nabla f|\|_\infty \leq 1$, and so we end up with 
\begin{align}\label{eq:equatorial-FEW-metastable}
\int Q_x (\tau_{|x_1| \geq  2 N^\epsilon} \leq T) d \pi (x \mid |x_1|\leq \tfrac 32 N^\epsilon) \lesssim T\cdot \exp\big(-N^{2\epsilon}\big)\,,
\end{align}
We now wish to boost this to an estimate on initialization from $\pi_0 = \pi_{0,\beta}$. As before, let $$\nu = \pi (\cdot\mid |x_1|\leq \tfrac{3}2 N^{\epsilon})\,, \qquad \mbox{and} \qquad \nu_{0} = \pi_0 (\cdot \mid |x_1|\leq \tfrac{3}{2} N^\epsilon)\,.$$

We begin by bounding the Radon--Nikodym derivative 
\begin{align*}
\exp\Big(- \beta (\tfrac{3}{2})^k  N^{ - \frac{k-2}{2}+\alpha+\epsilon k}\Big) \leq \max_{x: |x_1|\leq \frac{3}{2} N^\epsilon }  \frac{d\nu(x)}{d\nu_0(x)} \leq  \exp\Big( \beta (\tfrac{3}{2})^k  N^{ - \frac{k-2}{2}+\alpha+\epsilon k}\Big)
\end{align*}
but as soon as $\epsilon$ is sufficiently small, the power of $N$ is negative since $\alpha<\alpha_c(\infty)$ and thus, for $N$ sufficiently large, the Radon--Nikodym derivative satisfies $1/2\leq \|d\nu/d\nu_0\|_\infty\leq 2$. As a result, 
\begin{align*}
\int Q_x (\tau_{|x_1|\geq 2N^\epsilon} \leq T)d\nu_0 (x) \lesssim 2T \exp (-N^{2\epsilon})\,.
\end{align*}

Now notice that for some universal $C>0$, 
\begin{align*}
\pi_0(|x_1|\geq \tfrac{3}{2} N^\epsilon) \leq C\exp(-N^{2\epsilon}/C)\,;
\end{align*}
this follows as in the $\ell=0$ case of $\pi_0$ satisfying \CI~at level $n$, namely by concentration of $x_1$ under $dx$, and Markov's inequality.
Thus, writing 
\begin{align*}
\int Q_x (\tau_{|x_1|\geq 2N^\epsilon} \leq T) d\pi_0(x)\leq \pi_0 (|x_1|\geq \tfrac{3}{2}N^{\epsilon}) + 2\int Q_x (\tau_{|x_1|\geq 2N^\epsilon} \leq T) d\nu_0(x)
\end{align*} 
and plugging in the two inequalities above, we obtain the desired inequality by choosing $T=e^{cN^{\epsilon}}$ for $c>0$ sufficiently small.   \qed

\bibliographystyle{plain}
\bibliography{boundingflows}

\begin{thebibliography}{10}

\bibitem{abbe2018proof}
Emmanuel Abbe and Colin Sandon.
\newblock Proof of the achievability conjectures for the general stochastic
  block model.
\newblock {\em Communications on Pure and Applied Mathematics},
  71(7):1334--1406, 2018.

\bibitem{ABA13}
Antonio Auffinger and G{\'e}rard Ben~Arous.
\newblock Complexity of random smooth functions on the high-dimensional sphere.
\newblock {\em Ann. Probab.}, 41(6):4214--4247, 2013.

\bibitem{ABC13}
Antonio Auffinger, G{\'e}rard Ben~Arous, and Ji{\v{r}}{\'{\i}} {\v{C}}ern{\'y}.
\newblock Random matrices and complexity of spin glasses.
\newblock {\em Comm. Pure Appl. Math.}, 66(2):165--201, 2013.

\bibitem{Baik2005phase}
Jinho Baik, G{\'e}rard~Ben Arous, Sandrine P{\'e}ch{\'e}, et~al.
\newblock Phase transition of the largest eigenvalue for nonnull complex sample
  covariance matrices.
\newblock {\em The Annals of Probability}, 33(5):1643--1697, 2005.

\bibitem{barak2016nearly}
Boaz Barak, Samuel~B Hopkins, Jonathan Kelner, Pravesh Kothari, Ankur Moitra,
  and Aaron Potechin.
\newblock A nearly tight sum-of-squares lower bound for the planted clique
  problem.
\newblock In {\em Foundations of Computer Science (FOCS), 2016 IEEE 57th Annual
  Symposium on}, pages 428--437. IEEE, 2016.

\bibitem{barak2016noisy}
Boaz Barak and Ankur Moitra.
\newblock Noisy tensor completion via the sum-of-squares hierarchy.
\newblock In {\em Conference on Learning Theory}, pages 417--445, 2016.

\bibitem{BGJ18a}
G{\'e}rard Ben~Arous, Reza Gheissari, and Aukosh Jagannath.
\newblock Bounding flows for spherical spin glass dynamics.
\newblock {\em Preprint available on arXiv.}, 2018.

\bibitem{BAJag17}
G{\'e}rard Ben~Arous and Aukosh Jagannath.
\newblock Spectral gap estimates in mean field spin glasses.
\newblock {\em Communications in Mathematical Physics}, 361(1):1--52, 2018.

\bibitem{BMMN17}
Gerard Ben~Arous, Song Mei, Andrea Montanari, and Mihai Nica.
\newblock The landscape of the spiked tensor model.
\newblock {\em arXiv preprint arXiv:1711.05424}, 2017.

\bibitem{BGM12}
F.~Benaych-Georges, A.~Guionnet, and M.~Maida.
\newblock Large deviations of the extreme eigenvalues of random deformations of
  matrices.
\newblock {\em Probability Theory and Related Fields}, 154(3):703--751, Dec
  2012.

\bibitem{berthet2013optimal}
Quentin Berthet and Philippe Rigollet.
\newblock Optimal detection of sparse principal components in high dimension.
\newblock {\em The Annals of Statistics}, 41(4):1780--1815, 2013.

\bibitem{capitaine2009}
Mireille Capitaine, Catherine Donati-Martin, Delphine F{\'e}ral, et~al.
\newblock The largest eigenvalues of finite rank deformation of large wigner
  matrices: convergence and nonuniversality of the fluctuations.
\newblock {\em The Annals of Probability}, 37(1):1--47, 2009.

\bibitem{Ch17}
Wei-Kuo {Chen}.
\newblock Phase transition in the spiked random tensor with rademacher prior.
\newblock {\em arXiv preprint arXiv:1712.01777}, 2017.

\bibitem{EthierKurtz86}
Stewart~N. Ethier and Thomas~G. Kurtz.
\newblock {\em Markov processes}.
\newblock Wiley Series in Probability and Mathematical Statistics: Probability
  and Mathematical Statistics. John Wiley \& Sons, Inc., New York, 1986.
\newblock Characterization and convergence.

\bibitem{david2017high}
David Gamarnik and Ilias Zadik.
\newblock High dimensional regression with binary coefficients. estimating
  squared error and a phase transtition.
\newblock In {\em Conference on Learning Theory}, pages 948--953, 2017.

\bibitem{GeMa17}
Rong Ge and Tengyu Ma.
\newblock On the optimization landscape of tensor decompositions.
\newblock In I.~Guyon, U.~V. Luxburg, S.~Bengio, H.~Wallach, R.~Fergus,
  S.~Vishwanathan, and R.~Garnett, editors, {\em Advances in Neural Information
  Processing Systems 30}, pages 3653--3663. Curran Associates, Inc., 2017.

\bibitem{GJ16}
Reza Gheissari and Aukosh Jagannath.
\newblock On the spectral gap of spherical spin glass dynamics.
\newblock {\em Ann. Inst. H. PoincarŽ Probab. Statist.}, 55.(2):756--776, 05
  2019.

\bibitem{GiSh00}
Peter Gillin and David Sherrington.
\newblock $p > 2$ spin glasses with first-order ferromagnetic transitions.
\newblock {\em Journal of Physics A: Mathematical and General}, 33(16):3081,
  2000.

\bibitem{guillemin2010differential}
Victor Guillemin and Alan Pollack.
\newblock {\em Differential topology}, volume 370.
\newblock American Mathematical Soc., 2010.

\bibitem{HiLi13}
Christopher~J. Hillar and Lek-Heng Lim.
\newblock Most tensor problems are np-hard.
\newblock {\em J. ACM}, 60(6):45:1--45:39, November 2013.

\bibitem{hopkins2016fast}
Samuel~B Hopkins, Tselil Schramm, Jonathan Shi, and David Steurer.
\newblock Fast spectral algorithms from sum-of-squares proofs: tensor
  decomposition and planted sparse vectors.
\newblock In {\em Proceedings of the forty-eighth annual ACM symposium on
  Theory of Computing}, pages 178--191. ACM, 2016.

\bibitem{hopkins2015tensor}
Samuel~B Hopkins, Jonathan Shi, and David Steurer.
\newblock Tensor principal component analysis via sum-of-square proofs.
\newblock In {\em Conference on Learning Theory}, pages 956--1006, 2015.

\bibitem{Hsu02}
Elton~P. Hsu.
\newblock {\em Stochastic analysis on manifolds}, volume~38 of {\em Graduate
  Studies in Mathematics}.
\newblock American Mathematical Society, Providence, RI, 2002.

\bibitem{johnstone2000distribution}
Iain Johnstone.
\newblock {\em On the distribution of the largest principal component}.
\newblock Department of Statistics, Stanford University, 2000.

\bibitem{kim2017community}
Chiheon Kim, Afonso~S Bandeira, and Michel~X Goemans.
\newblock Community detection in hypergraphs, spiked tensor models, and
  sum-of-squares.
\newblock In {\em Sampling Theory and Applications (SampTA), 2017 International
  Conference on}, pages 124--128. IEEE, 2017.

\bibitem{Led01}
Michel Ledoux.
\newblock {\em The concentration of measure phenomenon}, volume~89 of {\em
  Mathematical Surveys and Monographs}.
\newblock American Mathematical Society, Providence, RI, 2001.

\bibitem{LMLKZ17}
Thibault Lesieur, L{\'e}o Miolane, Marc Lelarge, Florent Krzakala, and Lenka
  Zdeborov{\'a}.
\newblock Statistical and computational phase transitions in spiked tensor
  estimation.
\newblock In {\em Information Theory (ISIT), 2017 IEEE International Symposium
  on}, pages 511--515. IEEE, 2017.

\bibitem{maida2007}
Myl{\`e}ne Maida.
\newblock Large deviations for the largest eigenvalue of rank one deformations
  of gaussian ensembles.
\newblock {\em Electronic Journal of Probability}, 12:1131--1150, 2007.

\bibitem{mont15}
Andrea Montanari, Daniel Reichman, and Ofer Zeitouni.
\newblock On the limitation of spectral methods: From the gaussian hidden
  clique problem to rank-one perturbations of gaussian tensors.
\newblock In {\em Advances in Neural Information Processing Systems}, pages
  217--225, 2015.

\bibitem{Peche06}
S.~P{\'e}ch{\'e}.
\newblock The largest eigenvalue of small rank perturbations of hermitian
  random matrices.
\newblock {\em Probability Theory and Related Fields}, 134(1):127--173, Jan
  2006.

\bibitem{perry2016statistical}
Amelia Perry, Alexander~S Wein, and Afonso~S Bandeira.
\newblock Statistical limits of spiked tensor models.
\newblock {\em arXiv preprint arXiv:1612.07728}, 2016.

\bibitem{BMPW16}
Amelia Perry, Alexander~S Wein, Afonso~S Bandeira, and Ankur Moitra.
\newblock Optimality and sub-optimality of pca for spiked random matrices and
  synchronization.
\newblock {\em arXiv preprint arXiv:1609.05573}, 2016.

\bibitem{MR14}
Emile Richard and Andrea Montanari.
\newblock A statistical model for tensor pca.
\newblock In {\em Advances in Neural Information Processing Systems}, pages
  2897--2905, 2014.

\bibitem{BBCR18}
Valentina Ros, G{\'e}rard Ben~Arous, Giulio Biroli, and Chiara Cammarota.
\newblock Complex energy landscapes in spiked-tensor and simple glassy models:
  ruggedness, arrangements of local minima and phase transitions.
\newblock {\em arXiv preprint}, 2018.
\newblock https://arxiv.org/abs/1804.02686.

\bibitem{Schoen42}
Issac~J. Schoenberg.
\newblock Positive definite functions on spheres.
\newblock {\em Duke Math. J.}, 9:96--108, 1942.

\bibitem{StroockVaradhan06}
Daniel~W. Stroock and S.~R.~Srinivasa Varadhan.
\newblock {\em Multidimensional diffusion processes}.
\newblock Classics in Mathematics. Springer-Verlag, Berlin, 2006.
\newblock Reprint of the 1997 edition.

\bibitem{Sub15}
Eliran {Subag}.
\newblock {The complexity of spherical p-spin models - a second moment
  approach}.
\newblock {\em The Annals of Probability}, April 2015.

\bibitem{Vershynin}
Roman Vershynin.
\newblock {\em High--Dimensional Probability}.
\newblock Cambridge University Press (to appear), 2018.

\bibitem{zdeborova2016statistical}
Lenka Zdeborov{\'a} and Florent Krzakala.
\newblock Statistical physics of inference: Thresholds and algorithms.
\newblock {\em Advances in Physics}, 65(5):453--552, 2016.

\end{thebibliography}

\end{document}